\documentclass{amsart}

\newtheorem{thm}{Theorem}
\newtheorem{lem}[thm]{Lemma}
\newtheorem{prop}[thm]{Proposition}

\newtheorem{cjt}[thm]{Conjecture}

\theoremstyle{remark}
\newtheorem{rk}[thm]{Remark}

\usepackage{lipsum}
\usepackage{enumitem}
\usepackage{caption}
\usepackage{subcaption}

\usepackage{amsmath,amssymb}
\usepackage{mathtools}
\usepackage{bm}

\usepackage{graphicx}
\usepackage[margin=3.8cm]{geometry}

\usepackage[dvipsnames]{xcolor}

\def\Se{S^e}
\def\So{S^o}

\def\See{S^{ee}}
\def\Seo{S^{eo}}
\def\Soe{S^{oe}}
\def\Soo{S^{oo}}

\def\cSee{\cS^{ee}}
\def\cSeo{\cS^{eo}}
\def\cSoe{\cS^{oe}}
\def\cSoo{\cS^{oo}}

\def\cA{\mathcal{A}}
\def\cB{\mathcal{B}}
\def\cP{\mathcal{P}}
\def\cS{\mathcal{S}}
\def\cT{\mathcal{T}}
\def\cSe{\cS^e}
\def\cSo{\cS^o}

\def\bx{\bar{x}}
\def\by{\bar{y}}

\def\mue{\mu^e}
\def\muo{\mu^o}

\title[]{Enumeration of corner polyhedra and 3-connected Schnyder labelings}

\author[\'E. Fusy]{\'Eric Fusy}
\address{\'E.F.: CNRS/LIGM, UMR 8049, Universit\'e Gustave Eiffel, Marne-la-vall\'ee, France.}
\email{eric.fusy@u-pem.fr}

\author[E. Narmanli]{Erkan Narmanli}
\address{E.N.: LIX, \'Ecole Polytechnique, Palaiseau, France.}
\email{erkan.narmanli@normalesup.org}

\author[G. Schaeffer]{Gilles Schaeffer}
\address{G.S.: CNRS/LIX, UMR 7161, \'Ecole Polytechnique, Palaiseau, France.}
\email{schaeffe@lix.polytechnique.fr}

\begin{document}

\maketitle

\begin{abstract}
  We show that corner polyhedra and 3-connected Schnyder
  labelings join the growing list of planar structures that can be set in
  exact correspondence with (weighted) models of quadrant
  walks via a bijection due to Kenyon, Miller, Sheffield and Wilson.

  Our approach leads to a first polynomial time algorithm to count
  these structures, and to the determination of their exact
  asymptotic growth constants: the number $p_n$ of corner polyhedra
  and $s_n$ of 3-connected Schnyder labelings of size $n$ respectively
  satisfy $(p_n)^{1/n}\to 9/2$ and $(s_n)^{1/n}\to 16/3$ as $n$ goes
  to infinity.

  While the growth rates are rational, like in the case of previously
  known instances of such correspondences, the exponent of the
  asymptotic polynomial correction to the exponential growth does not
  appear to follow from the now standard Denisov-Wachtel approach, due
  to a bimodal behavior of the step set of the underlying tandem
  walk. However a heuristic argument suggests that these exponents are 
  $-1-\pi/\arccos(9/16)\approx -4.23$ for $p_n$ and
  $-1-\pi/\arccos(22/27)\approx -6.08$ for $s_n$, which would imply that the associated series are not D-finite.
\end{abstract}

\section{Introduction}\label{sec:intro}
This article is concerned with the enumerative properties of two
fascinating families of discrete geometric structures, corner polyhedra
and rigid orthogonal surfaces. Corner polyhedra,
see Figure~\ref{fig:corner_poly}(a), were introduced by Eppstein and
Mumford \cite{eppstein2010steinitz} who were interested in the
  possibility to give an elegant characterization \emph{\`a la
    Steinitz} of the graphs that can be realized as 1-skeleton for certain classes of orthogonal polyhedra. On the other hand, 
    rigid orthogonal surfaces, see Figure~\ref{fig:schnyder_lab}(a), were considered by Felsner
    \cite{felsner2003geodesic} in relation with the order dimension of
    3-polytopes.

It turns out that these geometric structures can be described in a
very similar way by certain underlying combinatorial structures,
\emph{polyhedral orientations} for corner polyhedra, see Figure~\ref{fig:corner_poly}(b), and ($3$-connected) 
\emph{Schnyder labelings} for rigid orthogonal surfaces, see Figure~\ref{fig:schnyder_lab}(b). As
illustrated by Figure~\ref{fig:table}, and as already partially
observed by several authors, \emph{e.g.} \cite{eppstein2010}, these
combinatorial counterparts are similar to those that were already observed for contact-systems of 
horizontal/vertical segments and for rectangular tilings.

After recalling in Section~\ref{models} the definition of polyhedral
orientations and Schnyder labelings, and how to recast them in terms
of bipolar orientations, we move on in Section~\ref{bij-tandem} to set
up exact correspondences with certain weighted bi-modal models of
so-called tandem quadrant walks via a bijection due to Kenyon, Miller, Sheffield
and Wilson \cite{kenyon2019bipolar}. The resulting bijections,
stated as Proposition~\ref{claim:poly} and Proposition~\ref{claim:schnyder}, allow
us to describe in Section~\ref{enum} polynomial time algorithms to
count these structures, and moreover to determine in
Theorem~\ref{thm:growth_rates} their exact asymptotic growth constants,
see also Figure~\ref{fig:table}. An exact enumeration formula is known for the case of Schnyder labelings on triangulations~\cite{bernardi2009intervals}; we show in Section~\ref{sec:triangul} that it can be recovered from our results. 

As can be observed in Figure~\ref{fig:table}, our results on the enumeration of corner polyhedra and 3-connected Schnyder labelings parallel
those for the number of plane bipolar orientations, and for the number of
transversal structures.  However the analysis is made more difficult
by the fact that the tandem walks that we have to deal with have a
bimodal behavior: the step set available at a current point depends on
the parity of the ordinate of this point. 
This puts the asymptotic analysis out of reach by current known methods:    
determining the asymptotic behaviour of our counting sequences would require an extension of  the approach by Denisov and Wachtel~\cite{denisov2015random} to a bimodal setting.  
Resorting to a plausible but conjectural version of their argument,  
we are nevertheless able to state
Conjecture~\ref{cj:asympt} on the polynomial corrections, which would
imply that the associated series are not D-finite.

A nice final touch on the
emerging global picture is the possibility to recast these results in terms
of colored pseudoline contact systems (before last line in Figure~\ref{fig:table}), as explained in Section~\ref{sec:contacts}. 
\begin{figure}
  \begin{center}\footnotesize
\begin{tabular}{|c|c|c|c|}
  \hline
 contact-system of horizontal/ & rectangular tilings  & corner polyhedra & rigid orthogonal\\
vertical segments (\cite{Tamassia})   &  (\cite{kant1997regular})  & (\cite{eppstein2010steinitz}, Sec~\ref{sec:intro})          & surfaces (\cite{felsner2003geodesic}, Sec~\ref{sec:intro})\\\hline 
  separating decompositions & transversal structures  & polyhedral orientations & Schnyder labelings\\
  (\cite{de1995bipolar})
  & (\cite{kant1997regular,fusy2009transversal})   & (\cite{eppstein2010steinitz}, Sec~\ref{defs}) & (\cite{felsner2000convex}, Sec~\ref{defs})\\\hline
  plane bipolar  & $T$-transverse bipolar &  $P$-admissible bipolar & 
  $S$-transverse  bipolar  \\
    orientations (\cite{de1995bipolar}) & orientations (\cite{Na20}) & orientations (Sec~\ref{bipolar}) & orientations (Sec~\ref{bipolar})\\\hline
tandem walks & $T$-admissible tandem & $P$-admissible tandem & $S$-admissible tandem\\
(\cite{Na20})  & walks (\cite{Na20}) &  walks (Sec~\ref{subsec:application}) &  walks (Sec~\ref{subsec:application})\\\hline
free bicolored & rigid bicolored & free tricolored  & rigid tricolored \\
contact-systems (Sec~\ref{sec:contacts}) & contact-systems (Sec~\ref{sec:contacts}) & contact-systems (Sec~\ref{sec:contacts}) & contact-systems (Sec~\ref{sec:contacts}) \\\hline
$\asymp 8^n$ (\cite{baxter2001dichromatic}) & $\asymp (27/2)^n$ (\cite{Na20}) & $\asymp (9/2)^n$  (Sec~\ref{enumasympt}) & $\asymp (16/3)^n$ (Sec~\ref{enumasympt})\\\hline
\end{tabular}
\end{center}
\caption{Four parallel families of structures. The successive rows indicate: 1) the incarnation as a geometric structure, 2) the incarnation as 
a family of decorated maps, 3) the model of (decorated) plane bipolar orientations, 4) the model of (decorated) tandem walks in the quadrant, 5) the model of contact-systems of curves, 6) the 
asymptotic growth rate.} 
\label{fig:table}
\end{figure}

\section{Presentation of the two models}\label{models}
\subsection{Definitions}\label{defs}
A \emph{planar map} (shortly hereafter, a map) is a connected multigraph embedded on the oriented sphere up to orientation-preserving homeomorphism.  It is \emph{rooted} by marking a corner, whose incident face is taken as the outer face in planar representations. Vertices and edges are called \emph{inner} or \emph{outer} depending on whether they are incident to the outer face or not. 
A map is called \emph{Eulerian} if its vertices have even degree, in which case the faces can be uniquely bicolored in light and dark faces so that the outer face is light (any edge has a dark face on one side and a light face on the other side). Dually, a map is  bipartite iff all faces have even degree, in which case the vertex bicoloration is unique, up to choosing the color of a given vertex.  

\begin{figure}
\begin{center}
\includegraphics[width=\textwidth]{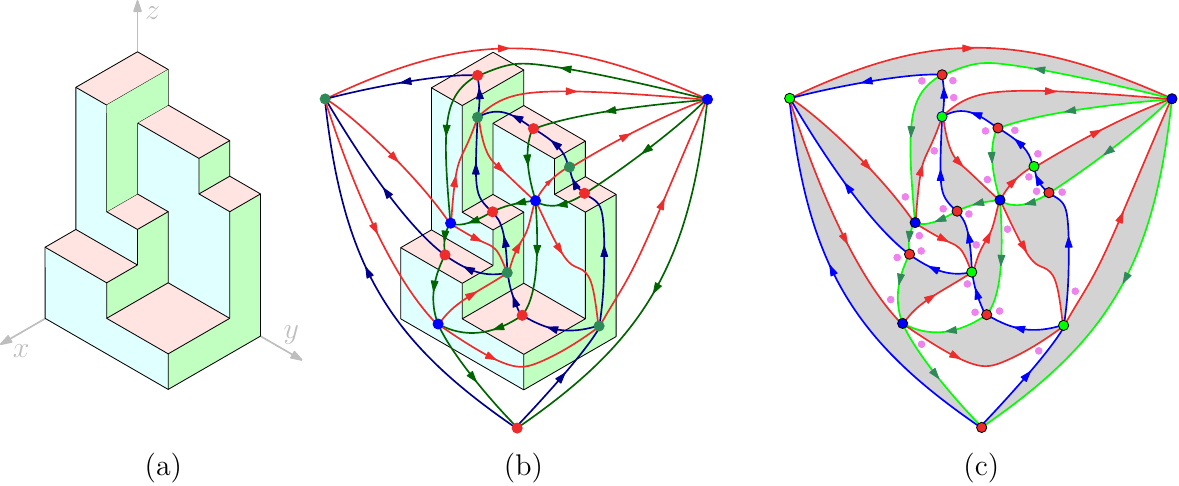}
\end{center}
\vspace{-1em}
\caption{From a corner polyhedron to an Eulerian triangulation endowed with a polyhedral orientation (extremal corners are indicated in violet, there are two such corners at each inner vertex and at each light inner face).}
\label{fig:corner_poly}
\vspace{-1em}
\end{figure}

A \emph{triangulation} is a planar map where all faces have degree $3$. It is known that a triangulation is 3-colorable iff it is Eulerian. In that case, the coloration of vertices (say in blue, green, red) is unique once the colors around a given triangle are fixed. If the triangulation is rooted, we take the convention that the root-vertex is red, and the outer vertices are colored red, green, blue in clockwise order around the outer face (i.e., walking along the outer contour with the outer face on the left). Note that every edge is also canonically colored red, green, or blue: it receives the color it misses (e.g. an edge connecting a green vertex and a blue vertex is colored red). 
In an orientation of a planar map, a corner $c=(v,e_1,e_2)$ is called \emph{lateral} if exactly one of $e_1,e_2$ is ingoing at $v$ (the other one being outgoing), it is called \emph{extremal} otherwise (either $e_1,e_2$ are both ingoing or both outgoing at $v$). 
For $T$ a rooted Eulerian triangulation, a \emph{polyhedral orientation} of $T$ is an orientation of  $T$ such that (see Figure~\ref{fig:corner_poly}(c) for an example):
\begin{itemize}
\item[(PO1):] There is no extremal corner at the outer vertices, and the outer contour is a cw cycle. 
\item[(PO2):] Every inner vertex is incident to exactly two extremal corners, and all the extremal corners are incident to light faces (hence dark face contours are either cw or ccw).
\end{itemize}
\begin{rk}\label{rk:counting_extremal}
Based on a counting argument, it can be checked that there must be exactly two extremal corners in every inner light face
(indeed, every inner light face has either zero or two extremal corners; but by the Euler relation the number of inner light faces equals the number of inner vertices,
so that the number of extremal corners is twice the number of inner light faces). \dotfill
\end{rk}
\begin{rk}
  Not every Eulerian triangulation admits a polyhedral orientation: in
  fact it is the case if and only if all its red/blue/green ccw triangles  are
  facial, as first shown in \cite{eppstein2010steinitz}. These so-called \emph{corner triangulations} have been enumerated in \cite{dervieux}. \dotfill
  \end{rk}
From now on, we call \emph{polyhedral orientation} a (corner) triangulation endowed with a polyhedral orientation. We let $\cP$ be the set of polyhedral orientations, and let $\cP_n$ (resp. $\cP_{a,b,c}$) be the set of polyhedral orientations with $n$ inner vertices (resp. with $a$ red inner vertices, $b$ blue inner vertices, and $c$ green inner vertices). 
Also we let $p_n=|\cP_n|$ and $p_{a,b,c}=|\cP_{a,b,c}|$ be the associated counting coefficients.

\begin{figure}[t]
\begin{center}
\includegraphics[width=\textwidth]{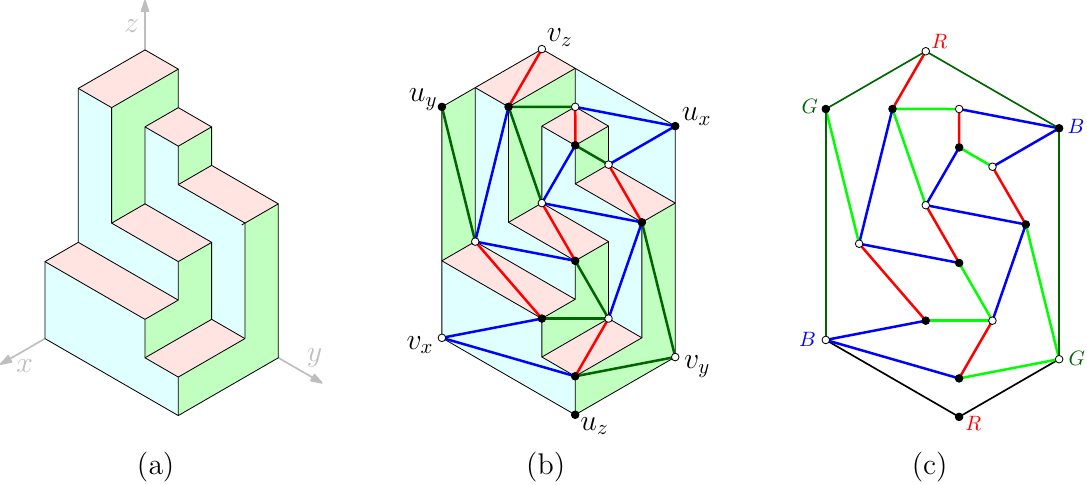}
\end{center}
\vspace{-1em}
\caption{From a rigid orthogonal surface to a $(6,4)$-dissection endowed with a Schnyder labeling (whose outer white vertices are non-isolated).}
\label{fig:schnyder_lab}
\end{figure}


\begin{figure}
  \begin{center}\small
\begin{tabular}{|c|c|c|}
  \hline
 & polyhedral orientations  & Schnyder labelings\\\hline\hline
 size  (univariate) & $n=\#$ inner vertices & $n=\#$ inner faces \\ \hline
 families & $\cP_n$ & \phantom{\Huge a}$\cS_n, \tilde{\cS}_n, \cS_n'$\phantom{\Huge a} \\ \hline
  coefficients & $p_n$ & $s_n, \tilde{s}_n, s_n'$\\\hline\hline
 size  (refinement) & $a,b,c=\#$ inner red,blue,green vertices & $a,b= -2+\#$ white,black vertices \\ \hline
 families & $\cP_{a,b,c}$ & \phantom{\Huge a}$\cS_{a,b}, \tilde{\cS}_{a,b}, \cS_{a,b}'$\phantom{\Huge a} \\ \hline
  coefficients & $p_{a,b,c}$ & $s_{a,b}, \tilde{s}_{a,b}, s_{a,b}'$\\\hline
\end{tabular}
\end{center}
\caption{Size-parameters and counting coefficients of polyhedral orientations and Schnyder labelings.} 
\label{fig:table2}
\end{figure}


A \emph{$(6,4)$-dissection} is a rooted map $D$ whose outer face is a simple cycle of length $6$, and whose inner faces have degree $4$. Such a map is bipartite, and the vertex-bicoloration (in black and white vertices) is the unique one such that the root-vertex is white. 
The outer vertices are labeled $R,G,B,R,G,B$ in ccw order around the outer face, starting with the root-vertex. An outer vertex is called \emph{isolated} if it has degree $2$ (i.e., is not incident to an inner edge). 
A \emph{Schnyder labeling} of $D$ is a coloration of the inner edges of $D$ in blue, green, red, such that (see Figure~\ref{fig:schnyder_lab}(c)):
\begin{itemize}
\item[(SL1):] The two outer vertices labeled $R$ (resp. $B, G$) have their incident inner edges red (resp. blue, green).
\item[(SL2):] 
The edges at each inner vertex form, in clockwise order, 3 non-empty groups of red, green and blue edges, respectively.  
\end{itemize}
\begin{rk}
  It is known~\cite{felsner2000convex,fusy2008dissections} that a $(6,4)$-dissection admits a Schnyder labeling iff it has no multiple edge and every $4$-cycle delimits a face. These dissections are counted  in~\cite{mullin1968enumeration,fusy2008dissections,bouttier2014irreducible}. \dotfill
  \end{rk}
\begin{rk}\label{rk:corner_map}
One can classically associate to $D$ a planar map $M$ with 3 distinguished outer vertices, which is obtained from $D$ by adding in each inner face an edge that connects the two opposite white vertices, and then erasing all edges and black vertices of $D$.
Via this mapping, our definition of Schnyder labelings matches the one of Felsner~\cite{felsner2000convex}, see~Figure~\ref{fig:felsner_polyhedra}. Precisely, he considers
those corresponding to the case where the $3$ outer white vertices are non-isolated. 
In our bijection for Schnyder labelings, we will rely on another subfamily, those where the two outer $G$ vertices are non-isolated. 
\dotfill
\end{rk}
\begin{figure}
  \includegraphics[width=\textwidth]{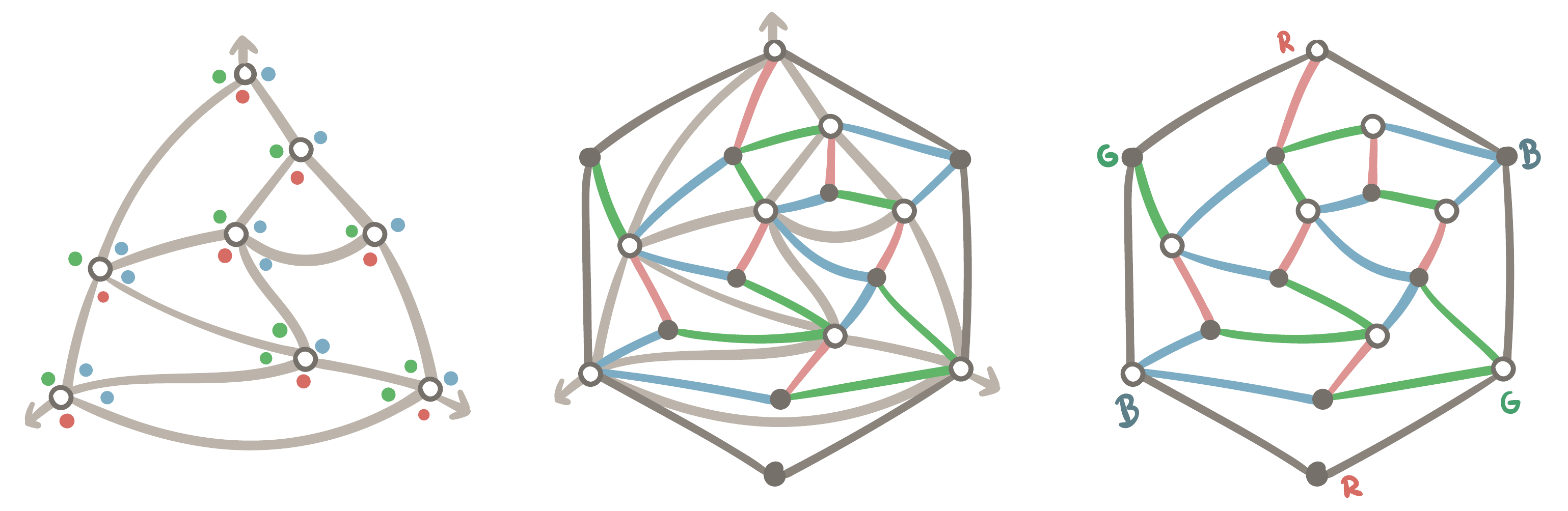}
  \caption{Correspondence between Schnyder labelings as considered by Felsner, on the left, and Schnyder labelings on $(6,4)$-dissections, on the right.}
  \label{fig:felsner_polyhedra}
\end{figure}


From now on, we call \emph{Schnyder labeling} a $(6,4)$-dissection endowed  with a Schnyder labeling. We let $\cS$ be the set of Schnyder labelings, let $\tilde{\cS}\subset\cS$ be the subset of those where the $3$ outer white vertices are non-isolated, and let $\cS'\subset\cS$ be the subset of those where the two outer $G$ vertices are non-isolated. 
We let $\cS_{n}$ (resp. $\tilde{\cS}_n$, $\cS'_n$) be the set of elements in $\cS$ (resp. $\tilde{\cS}$, $\cS'$) with $n$ inner faces, and let $s_n=|\cS_n|$ (resp. $\tilde{s}_n=|\tilde{\cS}_n|$, $s_n'=|\cS_n'|$) be the associated counting coefficients. 
As a  bivariate refinement, we let  $\cS_{a,b}$ (resp. $\tilde{\cS}_{a,b}$, $\cS_{a,b}'$) be the set of elements in $\cS$ (resp. $\tilde{\cS}$, $\cS'$) with $a+2$ white vertices and $b+2$ black vertices, and let $s_{a,b}=|\cS_{a,b}|$ (resp. $\tilde{s}_{a,b}=|\tilde{\cS}_{a,b}|$, $s_{a,b}'=|\cS_{a,b}'|$) be the associated counting coefficients. 
The Euler relation ensures that a $(6,4)$-dissection with $n$ inner faces has $n+4$ vertices, hence
$s_n=\sum_{a+b=n}s_{a,b}$ (resp. $\tilde{s}_n=\sum_{a+b=n}\ \tilde{s}_{a,b}$, $s_n'=\sum_{a+b=n}\ s_{a,b}'$). 

\begin{rk}\label{rk:colors_sch}
Based on a counting argument~\cite[Lem.1]{felsner2000convex}, it can be checked that, for $X\in\tilde{\cS}$, if the outer edges incident to the black vertex labeled $R$ (resp. $G$, $B$) are colored red (resp. green, blue), then the contour of every inner face has at least one edge in every color, with one of the three colors represented by two edges that are consecutive along the face contour. (The same holds for $X\in\cS'$, upon coloring red the 4 outer edges incident to an outer $R$ vertex, and blue the two outer $B-G$ edges.) 
\dotfill
\end{rk}

Figure~\ref{fig:table2} provides a summary of the notation for polyhedral orientations and Schnyder labelings.

  \subsection{Link to 3D representations}\label{sec:3drep}
  A \emph{simple orthogonal polyhedron}~\cite{eppstein2010steinitz}  is a 3D shape with the topology of a ball, and whose boundary is made of flat portions (called \emph{flats}), each flat being orthogonal to one of the $3$ coordinate axes.
  It is also required that if two flats share a boundary then their orthogonal directions are not the same, and that at most $3$ flats can intersect at a point on the boundary. The boundary yields a trivalent map (embedded on a topological 2D-sphere), 
  whose vertices correspond to the boundary points where $3$ flats meet, and whose faces correspond to the flats. A \emph{corner polyhedron} (see Figure~\ref{fig:corner_poly}(a) for an example)  is a simple orthogonal polyhedron having exactly $3$ flats whose orthogonal vector points in the coordinate-decreasing direction. These $3$ flats have to intersect, and the intersection is taken as the origin. The dual of the associated trivalent map is an Eulerian triangulation $T$ (a corner triangulation), whose outer face is taken as the one dual to the origin, with red (resp. blue, green) vertices associated to the flats whose orthogonal direction is the one of the $z$-axis (resp. $x$-axis, $y$-axis). Moreover, $T$ is naturally endowed with a polyhedral orientation, upon orienting every inner blue  (resp. green, red) edge in the direction of increasing $z$ (resp. $x$, $y$), see Figure~\ref{fig:corner_poly}(b). It is also shown in~\cite{eppstein2010steinitz} that every polyhedral orientation can be obtained in this way.  Polyhedral orientations can thus be considered as the combinatorial types of 
  corner polyhedra. Thus, $p_n:=|\cP_n|$ (resp. $p_{a,b,c}:=|\cP_{a,b,c}|$) gives the number of combinatorial types of corner polyhedra with $n+3$ facets (resp. with $a+b+c+3$ facets, among which $a+1$ are red, $b+1$ are blue, and $c+1$ are green), including the $3$ non-visible facets. 

  The union $\frak{S}$ of the boundary of a corner polyhedron $P$ and of the three 2D-quadrants 
  $Q_z=\{x,y\geq 0,z=0\},Q_x=\{y,z\geq 0,x=0\},Q_y=\{x,z\geq 0,y=0\}$ corresponds to a (non-degenerate, axial) orthogonal surface in the sense of~\cite{felsner2003geodesic}. There is another natural way to associate a decorated map to such a surface~\cite{felsner2003geodesic}.  
  Let $v_x$ (resp. $v_y,v_z$) be the extremal vertex of $P$ on the $x$-axis
  (resp. $y$-axis, $z$-axis). Let $u_x=v_y\vee v_z$, $u_y=v_x\vee v_z$ and $u_z=v_x\vee v_y$.  The \emph{white points} of $\frak{S}$ are $v_x,v_y,v_z$ and the points of $P$  (excluding the origin), called \emph{inward points of $P$}, where $3$ flats meet and the $3$ border rays leave the point in the coordinate-increasing way. The \emph{black points} of $\frak{S}$ are $u_x,u_y,u_z$ and the points of $P$, called \emph{outward points of $P$}, where $3$ flats meet and the $3$ border rays leave the point in the coordinate-decreasing way. 
  A black point $v_\bullet$ and a white point $v_\circ$ are \emph{adjacent} if $v_\bullet\geq v_\circ$ coordinatewise. 
  
  This adjacency relation 
  yields a bipartite graph, whose vertices are the black and white points, that is drawn on $\frak{S}$ (with edges as segments). The surface $\frak{S}$ is called \emph{rigid} if the drawing is crossing-free (the edges can meet only at common extremities); the example of  Figure~\ref{fig:corner_poly}(a) is non-rigid (there would be a crossing in the blue flat on the right) whereas the one in Figure~\ref{fig:schnyder_lab}(a) is rigid. In that case,  
  every edge $e=\{u,v\}$ is such that $u,v$ share exactly one coordinate. Then~$e$ is considered blue (resp. green, red) if 
  $u$ and $v$ have same $x$-coordinate (resp. $y$-coordinate, $z$-coordinate).  Upon adding the outer hexagon $v_z,u_y,v_x,u_z,v_y,u_x$ (labeled $R,G,B,R,G,B$), this embedded graph exactly gives a Schnyder labeling in $\tilde{\cS}$, see Figure~\ref{fig:schnyder_lab}(b). 
  It is shown in~\cite{felsner2003geodesic,felsner2008schnyder} that every Schnyder labeling in $\tilde{\cS}$ can be obtained in this way. 
  Schnyder labelings in $\tilde{\cS}$ can thus be considered as the combinatorial types of rigid orthogonal surfaces. The coefficient $\tilde{s}_n:=|\tilde{\cS}_n|$ 
  gives the number of combinatorial types of rigid orthogonal surfaces that arise from corner polyhedra with $n+3$ flats.  The coefficient $\tilde{s}_{a,b}:=|\tilde{\cS}_{a,b}|$ gives the number of combinatorial types of rigid orthogonal surfaces that arise from corner polyhedra with $a-1$ inward points and $b-1$ outward points. 

\begin{rk}
It would also be possible to consider the bivariate refinement $\cP_{a,b}$ (where parameters would have the same meaning as in  $\tilde{S}_{a,b}$) and the trivariate refinement $\tilde{S}_{a,b,c}$ (where parameters would have the same meaning as in $\cP_{a,b,c}$). However, we feel that the trivariate refinement is more natural for polyhedral orientations (the $3$ parameters are intrinsic to the underlying corner triangulation, 
they do not depend on which polyhedral orientation the triangulation is endowed with). Similarly, the bivariate refinement seems more natural for Schnyder labelings (the $2$ parameters are intrinsic to the underlying $(6,4)$-dissection).   \dotfill
\end{rk}

\subsection{Encoding by (constrained, decorated) plane bipolar orientations}\label{bipolar}
A \emph{plane bipolar orientation} is a rooted map endowed with an acyclic orientation with a unique source $S$ at the root-vertex, and a unique sink $N$ incident to the outer face. It is known~\cite{de1995bipolar} that a plane bipolar orientation is characterized by the following local properties (for orientations with $S$ as a  
source and $N$ as a sink), illustrated in Figure~\ref{fig:bipolar_rules}:
\begin{itemize} 
\item[(B):] Apart from $\{S,N\}$, each vertex has two lateral corners (so the incident edges form two groups: ingoing and outgoing edges).
\item[(B'):] Each face (including the outer one) has two extremal corners, so that the contour is partitioned into a left lateral path and a right lateral path that share their origins and ends, which are called the \emph{bottom vertex} and \emph{top vertex} of the face.
\end{itemize}
 The \emph{type} of a face is the integer pair $(i,j)$ such that the left (resp. right) lateral path of the face has length $i+1$ (resp. $j+1$). The \emph{outer type} of the orientation is the type of the outer face.
\begin{figure}[t]
  \begin{center}
    \includegraphics[width=\textwidth]{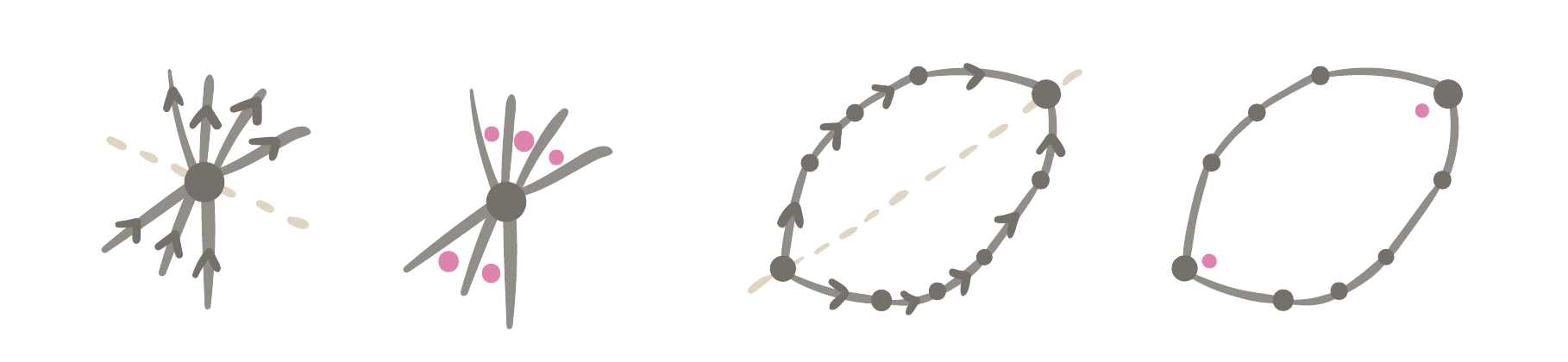}
  \end{center}
\vspace{-1em}
    \caption{Local rules for plane bipolar orientation, (B) on the left and~(B') on the right, and their translation in terms of lateral and extremal corners.}
    \label{fig:bipolar_rules}
\vspace{-1em}
\end{figure}
If the underlying map of the orientation is bipartite, i.e., the type $(i,j)$ of every inner face is such that $i+j$ is even, then the vertex bicoloration is chosen such that $N$ is white. An inner face is called a \emph{blacktip face} (resp. \emph{whitetip face}) if its top vertex is black (resp. white). 

A bipartite plane bipolar orientation is called \emph{$P$-admissible} iff
 \begin{itemize}
 \item[(PA1):] it has outer type $(0,k)$ for some even~$k\geq 2$, 
 \item[(PA2):]  the type $(i,j)$ of every blacktip (resp. whitetip) inner face is such that $i\geq 1$ (resp. $j\geq 1$).
 \end{itemize}

\begin{prop}\label{claim:poly}
For $n\geq 1$, polyhedral orientations with $n$ inner vertices, among which $a$ are red, $b$ are blue, and $c$ are green, are in bijection with $P$-admissible bipolar orientations with~$n+1$ edges, $a+1$ white vertices, $b+1$ black vertices, and $c$ inner faces.
\end{prop}
  \begin{proof}
    From a polyhedral orientation~$P$, its image~$X$ is obtained by removing all the green vertices, and recoloring the blue vertices
    as black, and the red vertices as white, see Figure~\ref{fig:poly_bijection}.   
   Clearly, $X$ is bipartite, and it forms a plane bipolar orientation~\cite[Lem 18]{eppstein2010steinitz} whose source (resp. sink) corresponds to the outer blue (resp. red) vertex of $P$.
Moreover, $X$ has outer type $(0,d-2)$, with $d\geq 4$ the degree of the green outer vertex of $P$. Hence, (PA1) holds for~$X$. 
Note that every corner $c$ in~$X$ correspond to exactly two corners in~$P$, respectively in a light and in a dark triangle. The light one is called
the \emph{attached corner} of $c$. (This gives a 1-to-1 correspondence between the corners of $X$ and the light corners at blue or red vertices in $P$.)  
Moreover, since the dark corners are always lateral, a corner in $X$ is extremal iff its attached corner is lateral. 
For $g$ a green inner vertex of $P$, with $f$ the corresponding inner face of $X$, 
the corner attached to the top corner (resp. bottom corner) of $f$ has thus to be lateral. Hence, by Remark~\ref{rk:counting_extremal},
the corner at $g$ in the same light triangle has to be extremal. Thus, the two extremal corners of $g$ are those that are  
in the light triangle $t_{\mathrm{up}}$ (resp. $t_{\mathrm{down}}$) touching the top-vertex (resp. the bottom-vertex) of $f$. 
The fact that these light triangles are different easily ensures that the left contour of $f$ contains an edge from a black to a white vertex, and the right contour of $f$
contains an edge from a white to a black vertex (see Figure~\ref{subfig:p-admissible_face} showing
the $4$ cases, depending on the colors of the top-vertex and bottom-vertex of $f$). This is equivalent to Condition (PA2) being satisfied by $f$. Hence, $X$ is $P$-admissible.

    \begin{figure}[t]
        \includegraphics[width=\textwidth]{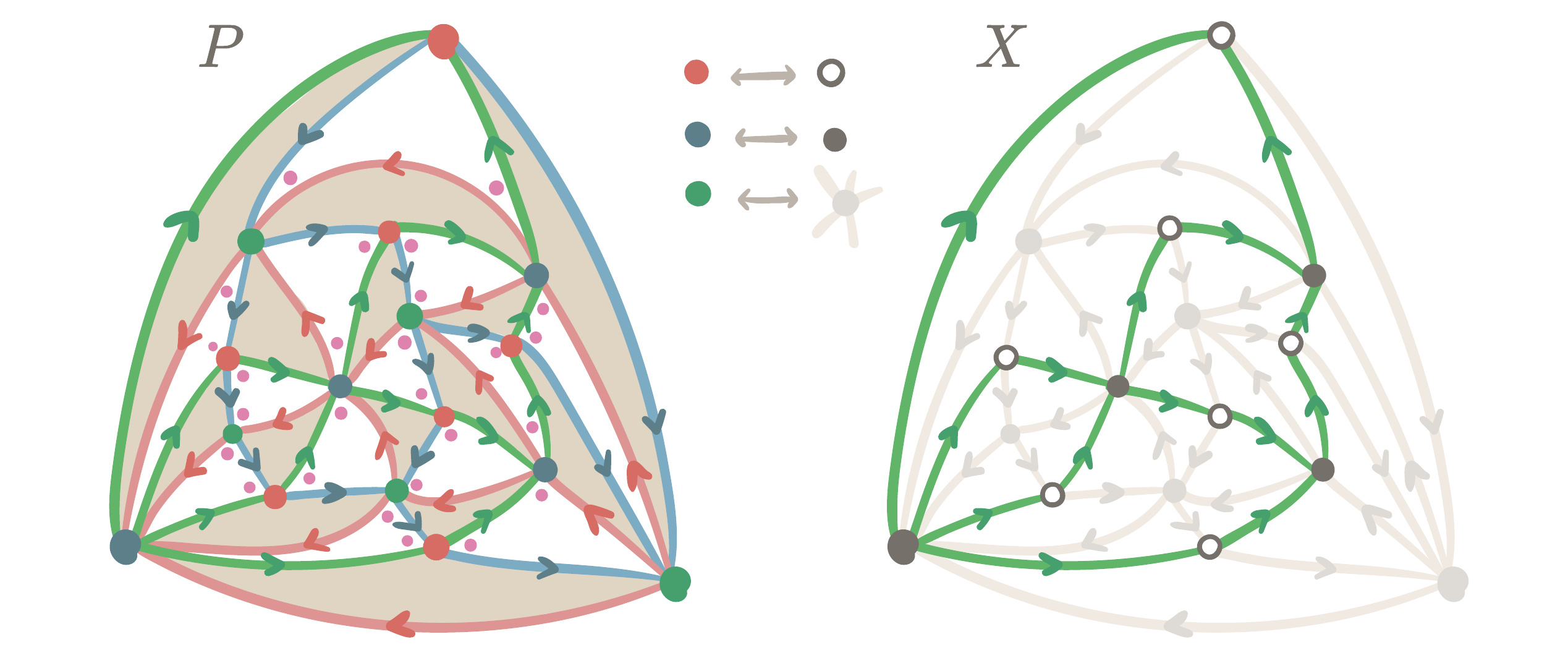}
        \vspace{-1em}
        \caption{On the left, a polyhedral orientation (with dots at extremal corners). On the right, the corresponding~$P$-admissible bipolar orientation.}
        \label{fig:poly_bijection}
        \vspace{-1em}
    \end{figure}

    \begin{figure}[t]
      \centering
      \includegraphics[width=0.8\textwidth]{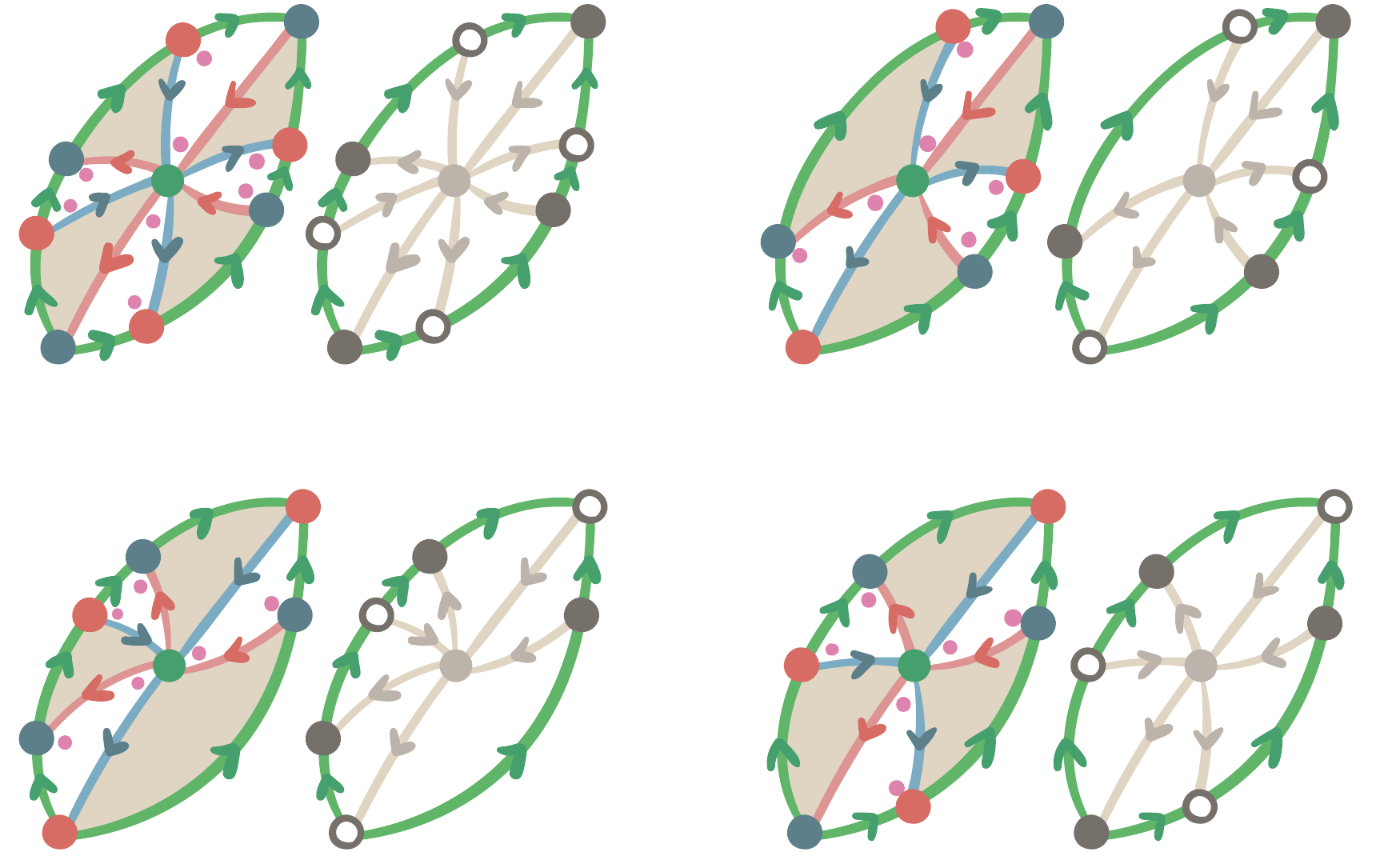}
      \caption{The correspondence for an inner green vertex $g$, associated to an inner face $f$ in the plane bipolar orientation, in the 4 different cases depending on the colors
      of the top-vertex and bottom-vertex of $f$.}
      \label{subfig:p-admissible_face}
    \end{figure}
        
Now, from a~$P$-admissible bipolar orientation~$X$, its image~$P$ is constructed as follows:
  \begin{enumerate}[label=(\roman*)]
      \item \label{it:2} Add a so-called face-vertex~$v_f$ in every face~$f$ of $X$, and connect $v_f$ to all corners around $f$.
      \item \label{it:3} Recolor all black (resp. white) vertices in blue (resp. in red); and all face-vertices in green.
      \item \label{it:4} Let~$v_{ext}$ be the face-vertex in the outer face of $X$: we orient its incident edges toward blue (former black) vertices and away from red (former white) vertices. In particular, the edge between $v_{ext}$ and the sink (resp. source) vertex of $X$ goes toward (resp. away from) $v_{ext}$ and edge directions alternate around $v_{ext}$.
      \item \label{it:5} For every inner face~$f$ of~$X$, we mark every corner attached to a lateral corner in~$X$ and we mark the two corners of~$v_f$ incident to the two light triangles  incident to the bottom and top corners of~$f$ (these two light triangles are distinct, due to (PA2)). Then there is exactly one way to orient blue and red edges so that 
      the extremal corners within $f$ are the marked ones, see Figure~\ref{subfig:p-admissible_face}.
  \end{enumerate}
  
  The bipartite nature of~$X$ and~\ref{it:3} ensures that~$P$ is~3-colorable, and~(PO1) follows from~\ref{it:4}. 
  By construction, all dark corners are lateral, and (PO2) is satisfied at all inner green vertices. 
  As for red and blue inner vertices, the construction is such that a corner of $X$ is extremal iff the attached corner in $P$ is lateral. Since each non-pole vertex of $X$
  has two lateral corners, the corresponding red or blue inner vertex in $P$ must have two extremal corners. Hence, (PO2) is satisfied at red and blue inner vertices. 

The two mappings are clearly inverse to each other, hence give a bijection.     
  \end{proof}

An \emph{$S$-transverse bipolar orientation} is a $(6,4)$-dissection $D$, where the edges are partitioned into \emph{plain edges} that are directed, and \emph{transversal edges} that are undirected, so that the following conditions are satisfied: 
\begin{itemize}
\item[(ST1):] Plain edges span all vertices of $D$, and form a (bipartite) plane bipolar orientation $X$ of outer type $(2,2)$, with at least one inner face. 
\item[(ST2):] Each transversal edge is within an inner face $f$ of $X$, and it connects a black vertex  in the interior of the left lateral path of $f$ and a white vertex in the interior of the  right lateral path of $f$. Moreover, in every inner face $f$, every black (resp. white) vertex in the interior of the left (resp. right) lateral of $f$ is incident to at least one transversal edge.  
\end{itemize}

\begin{prop}\label{claim:schnyder}
For $n\geq 1$, Schnyder labelings with $n$ inner faces, $a+2$ white vertices and $b+2$ black vertices ($n=a+b$), and whose two outer~$G$ vertices are non-isolated, are in bijection with $S$-transverse bipolar orientations with~$n+4$ vertices, among which $a+2$ are white and $b+2$ are black.
\end{prop}
\begin{proof}
The bijection is defined as follows. Given a Schnyder labeling $L$ whose $G$ outer vertices are non-isolated, 
let $s$ (resp. $t$) be the outer black (resp. white) vertex labeled $R$. 
 The \emph{left} (resp. \emph{right}) lateral path of $L$ is the path of outer edges from $s$ to $t$ with the outer face on its left (resp. right). 
 We color (red, blue, red) the edges on the left (resp. right) lateral path. We orient the red edges from black to white and the blue edges from white to black (in particular, the 
 left lateral path and right lateral path are directed from $s$ to $t$), and
 leave the green edges undirected. We claim that we obtain an $S$-transverse bipolar orientation, with source $s$ and sink $t$.  The main point is to check that the oriented map
 $X$ formed by the oriented edges (which spans all the vertices of $L$) 
 is acyclic. 
 
   \begin{figure}[t]
    \begin{center}
    \includegraphics[width=\textwidth]{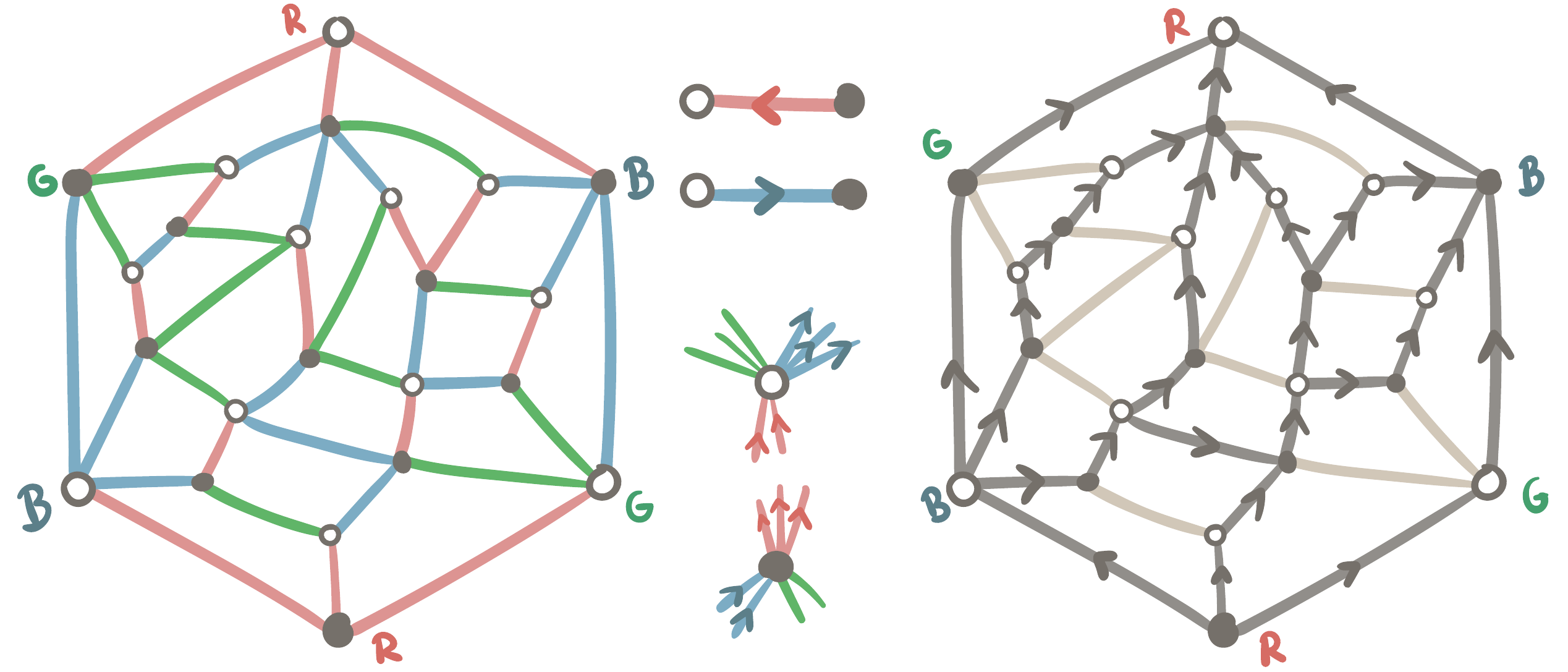}
    \end{center}
    \vspace{-1em} 
    \caption{On the left, a Schnyder labeling (with non-isolated outer G vertices), and on the right the corresponding~$S$-transverse bipolar orientation.}
    \label{fig:claim_schnyder_example}
  \vspace{-1em}
   \end{figure}
   
 Assume that $X$ has a directed cycle $\gamma$, and consider a minimal one, i.e., whose interior $\gamma^{\circ}$
 does not contain the interior of another directed cycle of $X$.   
The local conditions (SL1)-(SL2) imply that if $\gamma$ is clockwise (resp. counterclockwise), then any incidence of a transversal edge in $\gamma^\circ$ with 
a vertex $v$ on $\gamma$ must be such that $v$ is black (resp. white), and moreover every black (resp. white) vertex on $\gamma$ must be incident to at least one transversal edge in $\gamma^\circ$.  Hence, if $\gamma$ is clockwise (resp. counterclockwise), then there must be a transversal edge leaving a black vertex on $\gamma$ and whose other extremity $v$ is strictly in $\gamma^\circ$. Note that $s,t$ are the only extremal vertices of $X$, and they are exterior to $\gamma$. Hence, from $v$ starts a path of outgoing edges (the next edge at each step is an outgoing edge of the current vertex),   
which can not loop by minimality of $\gamma$, hence has to reach $\gamma$ at some vertex $v'$. Similarly, from $v$ starts a path of ingoing edges (the next edge at each step is an ingoing edge of the current vertex), which can not loop by minimality of $\gamma$, hence has to reach $\gamma$ at some vertex $v''\neq v'$. These two paths, together with  the  path 
on $\gamma$ connecting $v'$ to $v''$, form a directed cycle whose interior is in $\gamma^{\circ}$, a contradiction. 

Hence, $X$ is acyclic. Since it has a single source and a single sink (both incident to the outer face),
it is a plane bipolar orientation, and clearly its left and right outer paths are the left and right outer paths of $L$ defined above. Hence, $X$ has outer type $(2,2)$, and 
every transversal edge is within an inner face of $X$. Moreover, the local conditions (SL1)-(SL2)
 easily ensure that every transversal edge within an inner face $f$ of $X$ 
has to connect a black vertex in the interior of the left lateral path of $f$ to a white vertex in the interior of the 
right lateral path of $f$, and that every black (resp. white) vertex in the interior of the 
left (resp. right) lateral path of $f$ is incident to at least one transversal edge (the fact that in $L$ the two outer $G$ are non-isolated is necessary to have this property satisfied at these two vertices). 

The inverse mapping is defined as follows. Starting from an $S$-transverse bipolar orientation, with $X$ the part made by the plane bipolar orientation, 
we color green the transversal edges, and color red (resp. blue) the edges of $X$  
that are directed from black to white (resp. from white to black), and we forget the edge directions and the colors of the outer edges. The outer vertices of the obtained edge-colored
map $L$ are labeled $R,G,B,R,G,B$ in counterclockwise order around the outer face, starting from the sink of~$X$. 
 Condition (B) of plane bipolar orientations and condition (ST2) imply that condition (SL2) is satisfied,
and moreover that the two $G$ outer vertices are incident to at least one transversal edge, and every incidence of a transversal edge with an outer vertex must be with one of the two outer vertices labeled $G$. Condition (ST2) also implies that every inner face of $X$ has at least one black vertex in the interior of its left lateral path and at least one white vertex in the interior of its right lateral path. This easily implies that the left outer $G$ has indegree $1$ and outdegree $1$ in $X$, 
so that in $L$ all inner edges incident to the left outer $G$ (and similarly, the right outer $G$) are green.  Similarly, the left outer $B$ has indegree $1$ in $X$, so that
in $L$ all inner edges incident to this vertex (and similarly, to the outer $B$ on the right side) are blue. 
And all inner edges incident to one of the two outer vertices labeled $R$ (which are the source and sink of the bipolar orientation) are red. 
Hence, Condition (SL1) is satisfied, so that $L$ is a Schnyder labeling (whose two outer $G$ are non-isolated). 

Finally, the two mappings are clearly inverse to each other, hence give a bijection. 
\end{proof}

\begin{rk}\label{rk:constraint}
For enumerative purposes, the constraint that the two outer~$G$ vertices are non-isolated is mild.
Indeed, we have $s_n=s_n'+2s_{n-1}'+s_{n-2}'$ for $n\geq 4$ (the three terms correspond to having $0$, $1$, or $2$ isolated vertices among the two outer $G$ vertices).
Similarly, for bivariate enumeration, we have $s_{a,b}=s_{a,b}'+s_{a,b-1}'+s_{a-1,b}'+s_{a-1,b-1}'$ for $a\geq 2$ and $b\geq 2$. 
 \dotfill
\end{rk}

\begin{rk}\label{rk:transvers}
Let $X$ be an $S$-transverse bipolar orientation, and let $f$ be an inner face of $X$, with $q_0,\ldots,q_{m+1}$ the
quadrangular faces within $f$, ordered from bottom to top. Let $\gamma$ be the path from the first to the last black vertex
on the strict left boundary of $f$ (i.e., excluding the extremal vertices of $f$), and let $2\ell$ be its length. Let $\gamma'$ be the path from the first to the last white vertex
on the strict right boundary of $f$, and let $2r$ be its length.  
It is easy to see that for $h\in[1..m]$, $q_h$ either has two edges on $\gamma$ and none on $\gamma'$, or the opposite. We can thus attach to $f$ a word in $\frak{S}(o^\ell \bar{o}^r)$ giving the types of $q_1,\ldots,q_m$ ($o$ if the face has two edges on $\gamma$, $\bar{o}$ otherwise). It completely encodes the configuration of transversal edges within $f$, and any such word is a valid encoding. Hence the configuration can be encoded by an integer in $[1..\binom{\ell+r}{r}]$. Note also that the degree of $f$ is $2\ell+2r+6$ in all cases (in particular, all inner faces of $X$ have degree at least $6$).
 \dotfill
 \end{rk}

\section{Bijections with walks in the quadrant}\label{bij-tandem}

Similarly as in~\cite{Na20}, once our models have been set in bijection to certain models of plane bipolar orientations, they can be set in bijection to specific quadrant walks by specializing a bijection due to Kenyon, Miller, Sheffield and Wilson (shortly called the KMSW bijection), which we use as a bijective black box. An example of the KMSW bijction is given in Figure~\ref{fig:exemple_kmsw}.


\subsection{KMSW bijection}

A \emph{tandem walk} is a walk on the lattice $\mathbb{Z}^2$, with steps in $\{(1,-1)\}\cup\{(-i,j)|\ i,j\geq 0\}$. A step that is not a SE step (i.e., a step of the form $(-i,j)$) is called a \emph{face-step}. 

\begin{thm}[\cite{kenyon2019bipolar}]
There is a bijection between plane bipolar orientations of outer type $(d,d')$ and tandem walks from $(0,d)$ to $(d',0)$ staying in the quadrant $\mathbb{N}^2$. For $X$ a plane bipolar orientations and $\pi$ the corresponding tandem walk, the number of edges of $X$ corresponds to one plus the length of $\pi$, each inner face of type $(i,j)$ in $X$ corresponds to a face-step $(-i,j)$ in $\pi$, and each non-pole vertex corresponds to a SE step of $\pi$. 
\end{thm}

\begin{rk}
The bijection is easy to specialize to the bipartite setting (we will use the bijection in this setting only). A plane bipolar orientation $X$ is bipartite iff in the corresponding walk, each face-step $(-i,j)$ is such that $i+j$ is even; such a tandem walk is called \emph{even}. Moreover, the non-pole white and black vertices of $X$ correspond to the SE steps that start at even $y$ and odd $y$, respectively (this is due to the property that the $y$ where the step starts indicates a path-length in $X$ between $N$ and the vertex corresponding to the step).   Similarly, whitetip inner faces and blacktip inner faces correspond to face-steps that start at even $y$ and at odd $y$, respectively, see Figure~\ref{fig:exemple_kmsw} for an example. \dotfill
\end{rk}
\begin{figure}[t]
    \begin{center}
        \includegraphics[width=0.8\textwidth]{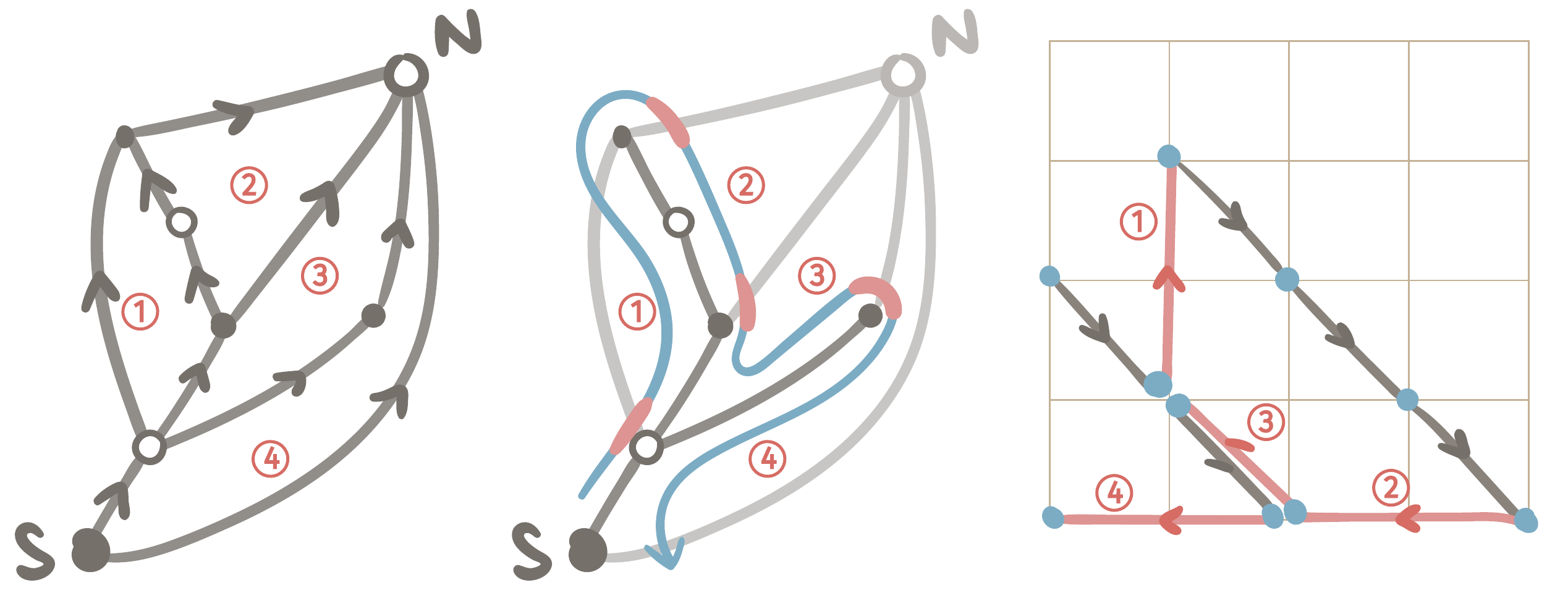}
    \end{center}
\vspace{-1em}
    \caption{A bipartite plane bipolar orientation (left), and its rightmost ingoing tree (middle), excluding the vertex~$N$. The corresponding even tandem walk (right) is obtained by revolving around the tree in clockwise order: when first seeing an edge, it yields a~$SE$ step, and when first entering a face of type $(i,j)$, it yields a face-step $(-i,j)$.}
    \label{fig:exemple_kmsw}
\vspace{-1em}
\end{figure}

\subsection{Application to the two models}\label{subsec:application}

We first specialize the KMSW bijection (in the bipartite setting) to the $P$-admissible bipolar orientations.  
A \emph{$P$-admissible tandem walk} is an even tandem walk where  every face-step $(-i,j)$ starting at even (resp. odd) $y$ has $j\geq 1$ (resp. $i\geq 1$). 
Via Proposition~\ref{claim:poly} we obtain:

\begin{prop}\label{prop:poly_tandem}
For $n\geq 1$, polyhedral orientations with $n$ inner vertices, among which $a$ are red, $b$ are blue, and $c$ are green, are in bijection with $P$-admissible tandem walks in the quadrant that start at the origin, end on the $x$-axis, and have $n$ steps, among which $a$ are SE steps starting at even $y$, $b$ 
are SE steps starting at odd $y$, and $c$ are face-steps. 
\end{prop}

We then specialize the KMSW bijection to the S-transverse bipolar orientations. For this (given Remark~\ref{rk:transvers}), we need a weighted terminology: a step $s$ in a tandem walk is said to be \emph{weighted} by $w\in\mathbb{N}$ if $s$ comes with an integer in $[1..w]$ (for the enumeration, the weights of the steps composing the walk have to be multiplied, those where no weight is indicated are implicitly assumed to have weight $1$). An \emph{$S$-admissible tandem walk} is defined as an even tandem walk such that every face-step $(-i,j)$ with even entries is of the form $i=2\ell+2,j=2r+2$ and is weighted by $\binom{\ell+r}{r}$,  every face-step $(-i,j)$ with odd entries and starting at even $y$ is of the form $i=2\ell+1,j=2r+3$ and is weighted by $\binom{\ell+r}{r}$, and  every face-step $(-i,j)$ with odd entries and starting at odd $y$ is of the form $i=2\ell+3,j=2r+1$ and is weighted by $\binom{\ell+r}{r}$.    
Via Proposition~\ref{claim:schnyder} and Remark~\ref{rk:transvers}, we obtain:

\begin{prop}\label{prop:schnyder-tandem}
For $n\geq 1$, Schnyder labelings with $n$ inner faces, $a+2$ white vertices and $b+2$ black vertices ($n=a+b$), and whose two outer~$G$ vertices are non-isolated, are in bijection with $S$-admissible tandem walks in the quadrant that start  at $(0,2)$, end at $(2,0)$, and have  $n+2$ SE steps, among which $a+1$ start at even $y$ and $b+1$ start at odd $y$. 
\end{prop}

\subsection{Specialization of Proposition~\ref{prop:schnyder-tandem} to triangulations}\label{sec:triangul}
It is known that Schnyder labelings on \emph{triangulations} with $n+3$ vertices are in bijection to non-crossing pairs of Dyck walks of length $2n$~\cite{bernardi2009intervals}. 
We explain
here how this bijective result can be recovered from Proposition~\ref{prop:schnyder-tandem}.  
We define a \emph{1-aligned} tandem walk as a tandem walk such that all face-steps $(-i,j)$ have $j=1$. Let $\cA_n$ be
the set of 1-aligned quadrant tandem walks starting at the origin and ending on the $x$-axis, with $n$ SE steps (note that the number of face-steps must also be $n$, and the length is $2n$). 

\begin{figure}
    \begin{center}
        \includegraphics[width=\textwidth]{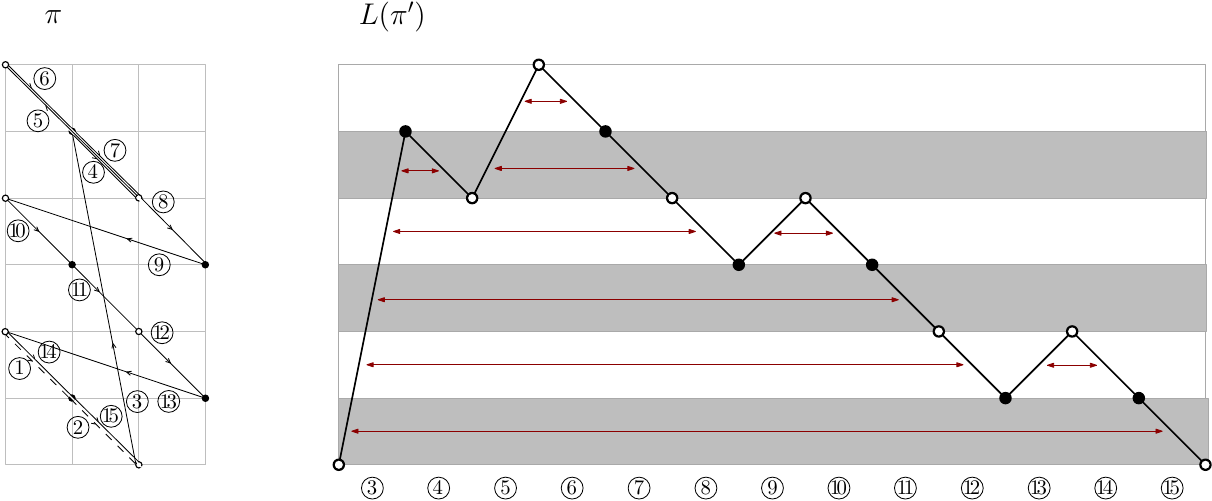}
    \end{center}
    \caption{An $S$-admissible quadrant tandem walk from $(0,2)$ to $(2,0)$ (without taking weights into account), and the corresponding Łukasiewicz excursion $L(\pi')$ (where $\pi'$ is $\pi$ without the first two steps). The initial 
    step $s$ of $L(\pi')$ has $a(s)=2$, $b(s)=3$, and $\delta(s)=2a(s)-b(s)=1$.}
    \label{fig:Luka}
\end{figure}
 
\begin{lem}\label{lem:bijS}
The family $\cS_{a,b}'$ is empty for $b>2a-1$. For $b=2a-1$, it is in bijection with $\cA_{a-1}$.
\end{lem}
\begin{proof}
By Proposition~\ref{prop:schnyder-tandem}, $\cS_{a,b}'$ is in bijection with the set of $S$-admissible quadrant tandem walks from $(0,2)$ to $(2,0)$ having $a+1$ (resp. $b+1$) SE steps starting at even (resp. odd) $y$.  Let $\pi$ be such a walk. 
 An easy case inspection ensures that the first two steps of $\pi$ 
have to be SE steps (as there is no step $(0,j)$, nor step $(-1,j)$ starting at odd $y$).  
Let $\pi'$ be $\pi$ deprived from the first two steps, hence $\pi'$ starts and ends at $(2,0)$.   
Let $L(\pi')$ be the (directed, 2d) walk with steps in $\{(1,j),\ j\geq -1\}$ obtained by the $y$-projection of $\pi'$, i.e., 
every face step $(-i,j)$ of $\pi'$ becomes a step $(1,j)$ in $L(\pi')$, and SE steps are unchanged. 
Clearly  $L(\pi')$ starts and ends on the $x$-axis and stays in the half-plane $\{y\geq 0\}$, i.e., it is a Łukasiewicz excursion; note also that the SE steps of $\pi'$
correspond to the down-steps of $L(\pi')$, while the face steps of $\pi'$ correspond to the rising steps of $L(\pi')$ (since $\pi$ is $S$-admissible, it has no horizontal  face step).    
 A step of $L(\pi')$ is called even or odd whether it starts at even or odd $y$. By the properties of $\pi$,     
 $L(\pi')$ has $a$ even downsteps and $b$ odd downsteps, and each rising step $s=(k,1)$ in $L(\pi')$ satisfies $k\geq 2$ if $s$ is even, and  $k\geq 1$ if $s$ is odd.
  
Note that each downstep of $L(\pi')$ from ordinate $y$ to $y-1$
is classically associated with the closest preceding rising step covering these ordinates, as illustrated in Figure~\ref{fig:Luka}. For a rising step $s$, we let $a(s)$ and $b(s)$ be respectively
the numbers of even and odd downsteps associated with $s$ (so that $|a(s)-b(s)|\leq1$), and we let $\delta(s)=2a(s)-b(s)$. 
Given the above condition on rising steps of $L(\pi')$, it is easy to see that $\delta(s)\geq 0$, and that $\delta(s)=0$ iff $s$ is an even
 $+3$ step. Moreover, an easy case inspection ensures that $\pi$ has to end with two SE steps, which have to be preceded by a step of the form $(-2i-3,1)$ or $(-2i-2,2)$ for some $i\geq 0$. The corresponding rising step $s$ in $L(\pi')$ has $a(s)=1,b(s)=0$ (and $\delta(s)=2$) in the first case, and $a(s)=1,b(s)=1$ (and $\delta(s)=1$) in the second case. 
 Let $\Delta$ be the sum of $\delta(s)$ over the rising steps of $L(\pi')$. Note that $\Delta=2a-b$. Moreover, the preceding discussion ensures that $\Delta\geq 1$ (i.e., $b$ must be at most $2a-1$, as claimed),     
 and $\Delta=1$ (extremal case $b=2a-1$) 
 iff all rising steps are even $+3$ steps, except for the last rising step that is an even $+2$ step starting from the horizontal axis.   
 We conclude that if $\pi$ is counted by $s_{a,2a-1}'$ then it starts with two SE steps, all its face-steps $(-i,j)$ start at even $y$,  and they have $j=3$ except for the last one that has $j=2$ to reach the point $(0,2)$ (before the final two SE steps to reach $(2,0)$).   
 Note also that all face-steps in such a walk must have weight $1$.    
  
From $\varpi\in\cA_{a-1}$, it is easy to produce such a walk $\pi$ as follows (we specify it by the sequence of steps, and the starting point, that has to be $(0,2)$): turn each SE step into two successive SE steps (combined, these have the effect of a step $(2,-2)=2*(1,-1)$), and turn each step $(-i,1)$ into two consecutive steps: a step $(-2i-1,3)$ followed by a SE step (combined, these have the effect of a step $(-2i,2)=2*(-i,1)$). Finally, if the walk ends at $(d',0)$, append the three steps $(-2d'-2,2), \mathrm{SE}, \mathrm{SE}$, prepend two SE steps, and choose the starting point at $(0,2)$. By construction, the mapping is such that every face-step $(-i,1)$ starting from $(x,y)$ in $\varpi$ becomes a face-step $(-2i-1,3)$ starting from $(2x+2,2y)$ in the corresponding walk $\pi$.  
It is also easy to check that the mapping is bijective (every walk with $\Delta=1$ is uniquely obtained in that way). 
\end{proof}

\begin{rk}\label{rk:schnyderT1}
Another way to see that $\cS_{a,b}'$ is empty for $b>2a-1$ is to argue via vertex-degrees. For $D\in\cS_{a,b}'$ every inner black vertex and the outer black vertex labeled $G$ have degree at least $3$, while the two other outer black vertices have degree at least $2$. 
Moreover, by the Euler relation, the number of edges is $2a+2b+3$. Hence, $3b+4\leq 2a+2b+3$, giving $b\leq 2a-1$. In the extremal case $b=2a-2$, all the inequalities have to be tight. Similarly, for $\cS_{a,b}$ the inequalities are the same, except that the black outer $G$ is allowed to have degree $2$. Hence, the extremal case is $b=2a$, in which case all inner black vertices have degree $3$, and the outer black vertices have degree $2$. There is an easy bijection from $\cS_{a,2a}$ to $\cS_{a,2a-1}'$ that  consists in removing the outer black vertex $G$ and its two incident edges.  Moreover, via the mapping mentioned in Remark~\ref{rk:corner_map}, $\cS_{a,2a}$ 
corresponds to the family $\cT_{a-1}$ of Schnyder labelings on triangulations with $a+2$ vertices.
\dotfill
\end{rk}

\begin{rk}\label{rk:schnyderT2}
The bijective link between $\cT_n$  and $\cA_n$ can also be established via the KMSW bijection. Indeed, the plane bipolar orientations corresponding to 1-aligned tandem walks ending on the $x$-axis are those  having a left outer boundary of length $1$, and such that every inner face has a right boundary of length $2$; and there is an easy bijection between $\cT_n$ and  such orientations with $n+2$ vertices~\cite[Sec.5]{fusy2009bijective}.  The resulting link between $\cT_n$  and $\cA_n$  via maps is more direct, but the interest of Lemma~\ref{lem:bijS} is that it establishes the result by sole inspection of the properties of $S$-admissible tandem walks. 
\dotfill
\end{rk}

A \emph{Dyck walk} of length $2n$ is a walk  from the origin to $(n,n)$ with steps in $\{E,N\}$, staying in the region $\{y\geq x\}$.  A pair $(D,D')$ of Dyck walks of length $2n$ is called \emph{non-crossing} if for every $y\in[0..n-1]$ the unique $N$ step starting at height $y$ in $D'$ is weakly left of the unique $N$ step starting at height $y$ in $D$. We now recover the following well-known bijective correspondence~\cite{bernardi2009intervals}:

\begin{figure}
    \begin{center}
        \includegraphics[width=0.8\textwidth]{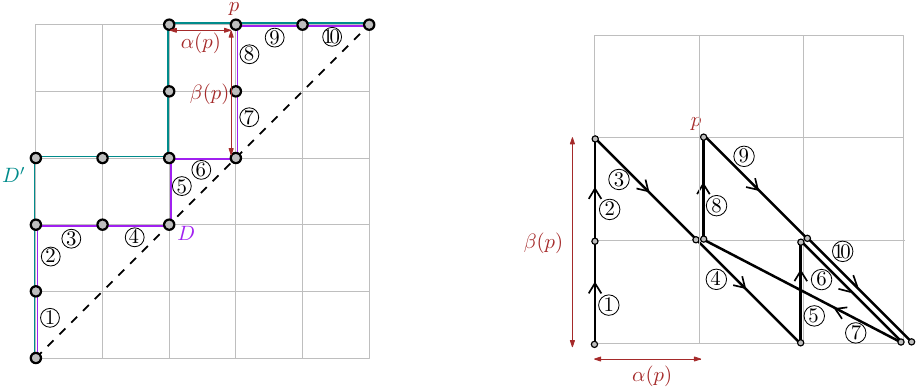}
    \end{center}
    \caption{A non-crossing pair of Dyck walks of length $2n$ ($n=5$), and the corresponding 1-aligned quadrant tandem walk in $\cA_n$.}
    \label{fig:dyck}
\end{figure}

\begin{prop}
For $n\geq 0$, Schnyder labelings on triangulations with $n+3$ vertices are in bijection
with non-crossing pairs of Dyck walks of length $2n$. 
\end{prop}
\begin{proof}
As we have seen in Remark~\ref{rk:schnyderT1}, $\cT_n$ identifies to 
 $\cS_{n+1,2n+1}'$. And this set is in bijection with $\cA_n$ (by Lemma~\ref{lem:bijS}, or alternatively Remark~\ref{rk:schnyderT2}).   
It remains to give a bijection between walks in $\cA_n$ and non-crossing pairs of Dyck walks of length $2n$.  The bijection is illustrated in Figure~\ref{fig:dyck}.   
For $\gamma=(D,D')$ a non-crossing pair of Dyck walks of length $2n$, and for each point $p=(x,y)$ on $D$ after the origin,
let $\alpha(p)$ be the horizontal distance between $p$ and the $N$ step of $D'$ arriving at height $y$, and let $\beta(p)$ be the vertical distance between $p$ and the diagonal $\{x=y\}$. With $p_1,\ldots,p_{2n}$ the sequence of points of $D$ after the origin, we define $\phi(\gamma)$ as the quadrant walk starting from the origin and visiting successively the points $(\alpha(p_i),\beta(p_i))$ for $i$ from $1$ to $2n$. Clearly, $\phi(\gamma)$ ends on the $x$-axis (since $\beta(p_{2n})=0$). It is also easy to see that each east step in $D$ yields a SE step in $\phi(\gamma)$, while
each north  step in $D$ yields a face-step that increases $y$ by $1$. Hence, $\phi(\gamma)$ is in $\cA_n$. 
The mapping $\phi$ is easy to invert, hence gives a bijection. 
\end{proof}

\section{Enumerative results}\label{enum}
\subsection{Exact enumeration}\label{enumexact}
A system of two equations with two catalytic variables $x$, $y$ can
easily be written for the series $Q^e(t,x,y)$ and
$Q^o(t,x,y)$ of $P$-admissible tandem walks staying in the quadrant, with even or odd final $y$
positions, along the lines for instance of \cite[Thm 3]{Beat20} (see also~\cite{buchacher2018inhomogeneous}), and
the same can be done for $S$-admissible tandem walks. The resulting
equations are however somewhat cumbersome to manipulate and it turns
out to be more efficient to reduce the problem to small step walk
problems, in the spirit of \cite[Prop. 4]{Na20}, but taking into
account the final $y$ parity.

We start with Schnyder labelings, which lead to simpler recurrences due to the weights on face-steps:
\begin{prop}\label{prop:exact_count1}
Let $s_n$ denote the number of Schnyder labelings with $n$ inner faces. 
Let moreover
$s^{\searrow}_n(i,j)$, and 
$s^{\nwarrow}_n(i,j)$ be given by the following recurrences:
\begin{align}
  s^{\searrow}_n(i,j)&=s^{\searrow}_{n-1}(i-1,j+1)+s^{\nwarrow}_{n-1}(i-1,j+1)\label{eq1s}\\
  s^{\nwarrow}_n(i,j)&=(s^{\searrow}_{n}(i+2,j-2)+s^{\nwarrow}_{n}(i+2,j-2))+(s^{\searrow}_{n}(i+1,j-3)\label{eq2s}\\
  &\quad +s^{\nwarrow}_{n}(i+1,j-3))+(s^{\nwarrow}_{n}(i+2,j)+s^{\nwarrow}_{n}(i,j-2))\quad\textrm{ if $j$ is odd, }\nonumber\\
    &=(s^{\searrow}_{n}(i+2,j-2)+s^{\nwarrow}_{n}(i+2,j-2))+(s^{\searrow}_{n}(i+3,j-1)\nonumber\\
  &\quad +s^{\nwarrow}_{n}(i+3,j-1))+(s^{\nwarrow}_{n}(i+2,j)+s^{\nwarrow}_{n}(i,j-2))\quad\textrm{ if $j$ is even, }\nonumber
\end{align}
with null boundary conditions for all coefficients $s^*_n(i,j)$ with $n\leq0$ or $i<0$ or $i>n$ or $j<0$ except $s^{\searrow}_0(0,2)=1$.

Then, for $n\geq 4$, $s_n=s'_{n}+2s'_{n-1}+s'_{n-2}$ where  $s'_n=s^{\searrow}_{n+2}(2,0)$. 
\end{prop}
\begin{proof}
As we have seen in Remark~\ref{rk:constraint}, for $n\geq 4$ we have $s_n=s'_{n}+2s'_{n-1}+s'_{n-2}$, where $s_n'=|\cS_n'|$ is the number of Schnyder labelings with $n$ inner faces, and where 
the two outer $G$ vertices are non-isolated.  
  According to Proposition~\ref{prop:schnyder-tandem}, $s_n'$ is the total weight of quadrant $S$-admissible tandem walks from $(0,2)$ to $(2,0)$ that have $n+2$ SE steps, for $n\geq 1$. 

  Observe now that a (weighted) $S$-admissible tandem walk identifies
  with an unweighted tandem walk with step set
  $\Sigma=\{(1,-1),(-2,2),(-3,1),(-1,3),(-2,0),(0,2)\}$, starting with a step in $\Sigma\backslash\{(-2,0),(0,2)\}$, such that
  $(-1,3)$ steps (resp. $(-3,1)$ steps) always start from an even
  (resp. odd) $y$ position, and $\{(-2,0),(0,2)\}$-steps never follow
  $(1,-1)$ steps. Indeed, the weight on a face-step from $(x,y)$ to $(x'=x-i,y'=y+j)$ in an
  $S$-admissible tandem walk exactly corresponds to the number of ways
  to convert such a step into a walk from $(x,y)$ to $(x',y')$ starting
  with a step in $\{(-2,2),(-1,3)\}$ (resp. $\{(-2,2),(-3,1)\}$) if $y$ is even (resp. odd), and followed with a
  sequence of steps in $\{(-2,0),(0,2)\}$. 

  For $n\geq 1$, let $s_n^\searrow(i,j)$ (resp. $s_n^\nwarrow(i,j)$)
  denote the number of $S$-admissible quadrant tandem walks starting at position
  $(0,2)$, ending at position $(i,j)$ and with $n$ SE steps, whose associated unweighted tandem walk with step set
  $\Sigma$ ends with a SE step (respectively with a step of $\Sigma\setminus\{(1,-1)\}$). A last step removal decomposition of these
  unweighted walks then directly yields Equations
  \eqref{eq1s}--\eqref{eq2s}, upon setting $s^*_n(i,j)=0$ for $n\leq 0$, except for
  $s^\searrow_0(0,2)=1$, to ensure propre initialization (the walk should not start with a step in $\{(-2,0),(0,2)\}$).
  
  Finally, observe that an $S$-admissible quadrant tandem walk ending on the
    horizontal axis has to finish with a SE step,  hence $s'_n=s^{\searrow}_{n+2}(2,0)$. 
\end{proof}

The recurrence allows us (thanks to the boundary conditions) to compute the first $n$ terms in polynomial time, using $O(n^3)$ additions on integers of size $O(n)$.  
The first terms are 
\[
\sum_{n\geq2} s_nt^n=3t^2+2t^3+3t^4+6t^5+14\,t^6+36\,t^7+102\,t^8+306\,t^9+972\,t^{10}+3216\,t^{11}+O(t^{12}).
\]

\medskip

Proposition~\ref{prop:exact_count1} can be refined to take into account the number of
black and white vertices, since these quantities respectively
correspond to the numbers of iterations through
Eq~\eqref{eq1s} at even or odd values of $j$. 
For $a,b\geq 2$ we have  $s_{a,b}=s_{a,b}'+s_{a,b-1}'+s_{a-1,b}'+s_{a-1,b-1}'$, where (for $n\geq 1$) 
 $\sum_{a+b=n}s_{a,b}'u^av^b=\frac1{uv}s^{\searrow}_{n+2}(2,0)$, with $s^{\searrow}_{n+2}(2,0)$ obtained from the same recurrence as in Proposition~\ref{prop:exact_count1}, except that the right-hand side in Eq~\eqref{eq1s}
is to be multiplied by $u$ (resp. $v$) if $j$ is odd (resp. even). 
The first terms are
\begin{equation*}
\begin{split}
\!\!\sum_{a,b\geq 1} \!\!s_{a,b}u^av^bt^{a+b}\!=&\ 3uvt^2\!+\!(u^2v\!+\!uv^2)t^3\!+\!3u^2v^2t^4\!+\!(3u^3v^2\!+\!3u^2v^3)t^5+\!( {u}^{4}{v}^{2}\!+\!12{u}^{3}{v}^{3}\!+\!{u}^{2}{v}^{4} ) {t}^{6} \\&+ ( 18{u}^{4}{v}^{3}+18{u}^{3}{v}^{4} )\, {t}^{7}+ ( 12{u}^{5}{v}^{3}+78{u}^{4}{v}^{4}+12{u}^{3}{v}^{5} )\, {t}^{8} + O(t^{9}).
\end{split}
\end{equation*}
Pushing further the expansion, we recognize that 
 the coefficient of $u^{a}v^{2a}t^{3a}$ matches the number $\mathrm{Cat}_{a+1}\mathrm{Cat}_{a-1}-\mathrm{Cat}_{a}\mathrm{Cat}_{a}$ of non-crossing pairs of Dyck walks of length $2a-2$,  
 as expected from Section~\ref{sec:triangul}. This sequence starts as $1, 1, 3, 14, 84, 594,\ldots$ \cite[A005700]{OEIS}.

\begin{rk}
We are also interested in counting
 Schnyder labelings whose outer white vertices are non-isolated, due to their link to rigid orthogonal surfaces, as discussed
 in Section~\ref{sec:3drep}.   
The counting sequence $\tilde{s}_n:=|\tilde{\cS}_n|$ can easily be obtained from the counting sequence $s_n$, by a similar argument as in Remark~\ref{rk:constraint}. The counting series is given by
\begin{align*}
\sum_{n\geq 2}\tilde{s}_nt^n&=\frac{1+3t+\sum_{n\geq 2}s_nt^n}{(1+t)^3}-1\\
&=t^3+3\,t^5+4\,t^6+15\,t^7+42\,t^8+131\,t^9+438\,t^{10} + 1467\,t^{11}+5204\,t^{12}+O(t^{13}).
\end{align*}

For bivariate enumeration, with $\tilde{s}_{a,b}:=|\tilde{\cS}_{a,b}|$, we have
\begin{align*}
&\sum_{a,b\geq 1}\tilde{s}_{a,b}u^av^bt^{a+b}=\frac{(1+3ut)v/u+\sum_{a,b\geq 1}s_{a,b}u^av^bt^{a+b}}{(1+ut)^3}-v/u\\
&\ \ \ \ \ \ \ =   {u}{v}^2{t}^{3}+3\,{u}^{2}{v}^{3}{t}^{5}+ ( 3\,{u}^{3}{v}^{3} +{u}^{2
}{v}^{4}) {t}^{6}+15\,{u}^{3}{v}^{4}{t}^{7}+
 ( 30\,{u}^{4}{v}^{4}+12\,{u}^{3}{v}^{5}) {t}^{8}+O(t^9).
\end{align*}
\end{rk}

\medskip

  Similarly, a recurrence can be obtained for polyhedral orientations,
  although 3 sequences are necessary, due to a further restriction on the
  set of admissible small steps walks to consider:
  \begin{prop}
  Let $p_n$ denote the number of polyhedral orientations with $n$ inner vertices.  Let moreover
  $p^{\searrow}_n(i,j)$,
  $p^{\leftarrow}_n(i,j)$, and
  $p^{\uparrow}_n(i,j)$ be given by the following recurrences:
  \begin{align}
    p^{\searrow}_n(i,j)&=p^{\searrow}_{n-1}(i-1,j+1)+p^{\nwarrow}_{n-1}(i-1,j+1)+p^{\uparrow}_{n-1}(i-1,j+1)\label{eq1p}\\
    p^{\nwarrow}_n(i,j)&=(p^{\searrow}_{n-1}(i+1,j-1)+p^{\nwarrow}_{n-1}(i+1,j-1)+p^{\uparrow}_{n-1}(i+1,j-1))\label{eq2p}\\
      &\quad+(p^{\searrow}_{n-1}(i,j-2)+p^{\nwarrow}_{n-1}(i,j-2)+p^{\uparrow}_{n-1}(i,j-2))\nonumber\\ 
      &\quad +(p^{\nwarrow}_{n}(i+2,j))\quad\textrm{if $j$ is even,}\nonumber\\
      &=(p^{\searrow}_{n-1}(i+1,j-1)+p^{\nwarrow}_{n-1}(i+1,j-1)+p^{\uparrow}_{n-1}(i+1,j-1))\nonumber\\
      &\quad+(p^{\searrow}_{n-1}(i+2,j)+p^{\nwarrow}_{n-1}(i+2,j)+p^{\uparrow}_{n-1}(i+2,j))\nonumber\\ 
      &\quad+(p^{\nwarrow}_{n}(i+2,j))\quad\textrm{if $j$ is odd,}\nonumber\\
    p^{\uparrow}_n(i,j)&=p^{\nwarrow}_{n}(i,j-2)+p^{\uparrow}_{n}(i,j-2)\label{eq3p}
  \end{align}
  with boundary conditions  $p^*_n(i,j)=0$ for all $n\leq0$ or $i<0$ or $j<0$ or $i>n$, except $p^{\searrow}_0(0,0)=1$.

  Then $p_n=\sum_{i\ge0}p^{\searrow}_{n}(i,0)$ for $n\geq 1$.
  \end{prop}
  \begin{proof}
By Proposition~\ref{prop:poly_tandem}, $p_n$ is the number of $P$-admissible tandem walks of length $n$ starting from the origin and ending on the $x$-axis. 
Our strategy is again to show that $P$-admissible tandem walks identify with  
    certain marked tandem walks on a well chosen small step set, namely
    $E=\{(1,-1),(-2,0),(0,2),(-1,1)\}$. In order to do that in absence
    of weighting, we break the symmetry arbitrarily and we decompose face
    steps as follows:
    \begin{itemize}
    \item each face step of the form $(-2\ell-1,2r+1)$ is mapped to a
      sequence starting with a marked step $(-1,1)$, followed with $\ell$ steps
      $(-2,0)$ and ended by $r$ steps $(0,2)$;
    \item each face step of the form $(-2\ell,2r+2)$ that starts at
      even ordinate is mapped to a sequence starting with a marked
      step $(0,2)$, followed by $\ell$ steps $(-2,0)$, and ending with
      $r$ steps $(0,2)$;
    \item each face step of the form $(-2\ell-2,2r)$ that starts at
      odd ordinate is mapped to a sequence starting with a marked
      step $(-2,0)$, followed by $\ell$ steps $(-2,0)$, and ending with
      $r$ steps $(0,2)$;
    \end{itemize}
    Observe that the decomposition of a face step always starts with
    its unique marked step. Therefore, a $P$-admissible tandem walk can be
    recovered from the concatenation of the image of its steps, and
    its length corresponds to the number of marked or $(1,-1)$ steps in
    its image.
    
    Now we observe that the subset of marked tandem walks on $E$ that
    correspond to $P$-admissible walks is characterized by the
    following local rules:
    \begin{itemize}
    \item a step $(1,-1)$ can follow any kind of step,
    \item a marked step of type $(-1,1)$ can follow any kind of step,
    \item a marked step of type $(0,2)$ (resp. $(-2,0)$) can follow any
      kind of step provided it starts at even (resp. odd) ordinate,
    \item an unmarked $(-2,0)$ step can only follow a marked step or an
      unmarked $(-2,0)$ step,
    \item an unmarked $(0,2)$ step can only follow a marked step or an 
      unmarked step  $(-2,0)$ or  an unmarked step $(0,2)$.
    \end{itemize}
    For $n\geq 1$, let then $p^\searrow_n(i,j)$ (resp. $p^\nwarrow_n(i,j)$,
    resp. $p^\uparrow_n(i,j)$) denote the number of $P$-admissible quadrant tandem 
    walks starting at position $(0,0)$, ending at position $(i,j)$, 
    whose associated marked tandem walk on $E$ has $n$ marked or
    $(1,-1)$ steps, and ends with a step $(1,-1)$ (resp. with a marked
    step or unmarked $(-2,0)$ step, resp. with a $(0,2)$ step).  The
    local rules then clearly imply the announced last step
    removal decomposition equations \eqref{eq1p}--\eqref{eq3p}, upon assuming $p^*_n(i,j)=0$ except for
    $p^\searrow_n(0,0)=1$, to ensure that the counted walks start with
    a marked step as expected.

    Finally, observe that a $P$-admissible quadrant tandem walk ending on the
    horizontal axis has to finish with a SE step, so that
    $p_n=\sum_{i\geq0}p_{n}^\searrow(i,0)$.
  \end{proof}
  
 The first terms, computed from the recurrence, are
  \[
  \sum_{n\geq 1} p_nt^n= t^3 + 3\,t^5 + 4\,t^6 + 15\,t^7 + 39\,t^8 + 122\,t^{9} + 375\,t^{10} + 1212\,t^{11}+ 3980\,t^{12} + O(t^{13}).
  \]
  
  \begin{rk}
   It is easy to see that if two rigid corner polyhedra yield the same Schnyder labeling, then they yield the same polyhedral orientation.  
This gives a mapping from $\tilde{\cS}_n$ to $\cP_n$, which is surjective (the surjectivity will appear clearly in Section~\ref{sec:contacts}, via a formulation in terms of tricolored contact-systems of curves). Hence,  $\tilde{s}_n\geq p_n$, which we observe on 
the initial terms (the coefficients start to differ from $n=8$). \dotfill
  \end{rk}
  
  \medskip
  
  Again the proposition can be refined, to take into account the number
  of red, blue, and green vertices: blue inner vertices and red inner 
  vertices correspond respectively to the numbers of iterations in
  Equation~\eqref{eq1p} at even or odd $j$ respectively. 
  Thus, 
  we have  $\sum_{a+b+c=n}p_{a,b,c}u^av^b=\sum_{i\ge0}p^{\searrow}_{n}(i,0)$, where $p^{\searrow}_{n}(i,0)$ is obtained from the same recurrence as in Proposition~\ref{prop:exact_count1}, except that the right-hand side in Eq~\eqref{eq1p} is to be multiplied by $u$ (resp. $v$) if $j$ is odd (resp. even). 
The first terms are
  
  \begin{equation*} 
    \begin{split}
    \sum_{a,b,c\geq1} p_{a,b,c}&u^av^bw^ct^{a+b+c}= uvwt^3 + (u^2v^2w+uv^2w^2+u^2vw^2)t^5 + 4u^2v^2w^2t^6 \\ &+ (u^{3} v^{3} w + 4 u^{3} v^{2} w^{2} + 4 u^{2} v^{3} w^{2} + u^{3} v w^{3} + 4 u^{2} v^{2} w^{3} + u v^{3} w^{3})t^7 + O(t^8).
  \end{split}
  \end{equation*}

\subsection{Asymptotic enumeration}\label{enumasympt}
Our main result regarding the asymptotic enumeration is to show that the growth rates of the coefficients $p_n$ and $s_n$ 
are respectively $9/2$ and $16/3$. We also conjecture in each case the exponent of the polynomial correction.  

\begin{prop}[upper bounds]\label{prop:bound}
The coefficients $p_n$ and $s_n$ satisfy the bounds 
$p_n\leq (9/2)^{n+1}$  and $s_n\leq 2\cdot (16/3)^{n}$. 
\end{prop}

\begin{proof}[Proof for~$p_n$.]
  A $P$-admissible tandem walk of length $n$ starting at the origin and staying in the quadrant can not pass again by the origin (the point preceding the origin would have to be on the $x$-axis, but no horizontal step is allowed for a point on the $x$-axis), hence, if it ends on the $x$-axis, it ends at a point of the form $(2i+2,0)$ for some $i\geq 0$. It is then allowed to add a further step $(-2i-1,1)$ to reach the point $(1,1)$. This operation being injective,  $p_n$ is bounded by the number of $P$-admissible tandem walks of length $n+1$ from the origin to $(1,1)$. 

  Let $\Se(x,y)$ (resp. $\So(x,y)$) be the step-series for steps starting at even $y$ (resp. odd $y$). 
  Starting from a point at even $y$ the allowed steps are in 
  $\cSe:=SE\cup \{(-2i,2j),\ i\geq 0, j\geq 1\} \cup \{(-2i-1,2j+1),\ i,j\geq 0\}$, 
  and starting from a point of odd $y$ the allowed steps are in $\cSo:=SE\cup \{(-2i,2j),\ i\geq 1, j\geq 0\} \cup \{(-2i-1,2j+1),\ i,j\geq 0\}$).  
  Hence,
  \[
  \Se(x,y)=x\by+\frac1{1-\bx^2}\frac{y^2}{1-y^2}+\frac{\bx}{1-\bx^2}\frac{y}{1-y^2},\ \ \So(x,y)=x\by+\frac{\bx^2}{1-\bx^2}\frac{1}{1-y^2}+\frac{\bx}{1-\bx^2}\frac{y}{1-y^2}.
  \]
  Note that $\Se(z^{-1/2},z^{1/2})=\So(z^{-1/2},z^{1/2})=z^{-1}+2z/(1-z)^2=:S(z)$, valid (as a converging sum) for $z\in(0,1)$. For $z\in(0,1)$ the quantity $S(z)^n$ thus gives the total weight of $P$-admissible tandem walks of length $n$ starting at the origin, where those ending at $(i,j)$ (for every $(i,j)\in\mathbb{Z}^2$) are weighted by $z^{(j-i)/2}$. Note that every walk ending at $(1,1)$ has weight $1$ in this series, so that $p_n\leq S(z)^{n+1}$. The best bound is obtained for $z\in(0,1)$ minimizing $S(z)$. One finds $S'(z)=0$ for $z_0=1/3$, with $S(z_0)=9/2$. Hence $p_n \leq (9/2)^{n+1}$.
\end{proof}

\begin{proof}[Proof for~$s_n$.]
We have $s_n=s'_n+2s_{n-1}'+s_{n-2}'$, where $s_n'$ is equal to the number of $S$-admissible quadrant tandem walks with $n+2$ SE steps, from $(0,2)$ to $(2,0)$. Such walks have to begin and end with a SE step. Hence, $s_n'$ is also the number of $S$-admissible quadrant tandem walks from $(0,2)$ to $(1,1)$ with $n+1$ SE steps, and all these walks begin with a SE step. Such a walk is thus a sequence of $n+1$ aggregated steps, where an aggregated step is a sequence of steps formed by a SE step followed a nonnegative number of  (weighted) face-steps, and satisfying the required conditions of $S$-admissible tandem walks.  
Let $\tilde{S}^{ee}(x,y)$ (resp. $\tilde{S}^{eo}(x,y)$) be the step series of face-steps in $S$-admissible tandem walks starting at a point at even $y$ and ending at a point at even (resp. odd) $y$, and taking the weight of the face-step into account.   
Similarly, let $\tilde{S}^{oe}(x,y)$ (resp. $\tilde{S}^{oo}(x,y)$) be the step series of face-steps in $S$-admissible tandem walks starting at a point of odd $y$ and ending at a point of even (resp. odd) $y$. 
We have
\begin{align*}
\tilde{S}^{ee}(x,y)&=\tilde{S}^{oo}(x,y)= \sum_{\ell,r\geq 0}\binom{\ell+r}{r}\bx^{2\ell+2}y^{2r+2}=\frac{\bx^2y^2}{1-\bx^2-y^2},\\
\tilde{S}^{eo}(x,y)&= \sum_{\ell,r\geq 0}\binom{\ell+r}{r}\bx^{2\ell+1}y^{2r+3}=\frac{\bx y^3}{1-\bx^2-y^2},\\
\tilde{S}^{oe}(x,y)&= \sum_{\ell,r\geq 0}\binom{\ell+r}{r}\bx^{2\ell+3}y^{2r+1}=\frac{\bx^3 y}{1-\bx^2-y^2}.
\end{align*}
Now let $A^e(x,y)$ (resp. $A^o(x,y)$) be the series for a possibly empty aggregation of face-steps (not including an initial SE step), starting at even (resp. odd) $y$, in an $S$-admissible tandem walk.   These are specified by the system
\[
\left\{
\begin{array}{ll}
A^e(x,y)&=1+\tilde{S}^{ee}(x,y)A^e(x,y)+\tilde{S}^{eo}(x,y)A^o(x,y),\\
A^o(x,y)&=1+\tilde{S}^{oe}(x,y)A^e(x,y)+\tilde{S}^{oo}(x,y)A^o(x,y),
\end{array}
\right.
\]
which gives
\[
A^e(x,y)=\frac{x^2y^2 - xy^3 - x^2 + y^2 + 1}{x^2y^2 - x^2 + 2y^2 + 1},\ \ A^o(x,y)=\frac{x^3y^2 - x^3 + xy^2 + x - y}{x(x^2y^2 - x^2 + 2y^2 + 1)} .
\]
and then the step-series $S^e(x,y)$ and $S^o(x,y)$ for an aggregated step starting from even (resp. odd) $y$ are given by  
\[
S^e(x,y)=x\by A^o(x,y),\ \ \ S^o(x,y)=x\by A^e(x,y).
\]

Under the specialization $(x=z^{-1/2},y=z^{1/2})$, a simplification occurs (indeed, as for $P$-admissible walks, the parity influence on the step choice is not visible when projecting to the axis $\{y=-x\}$). We have 
\[
\tilde{S}^{ee}(z^{-1/2},z^{1/2})=\tilde{S}^{eo}(z^{-1/2},z^{1/2})=\tilde{S}^{oe}(z^{-1/2},z^{1/2})=\tilde{S}^{oo}(z^{-1/2},z^{1/2})=\frac{z^2}{1-2z},
\]
and $A^e(z^{-1/2},z^{1/2})=A^o(z^{-1/2},z^{1/2})=\frac{1}{1-2z^2/(1-2z)}$, so that
\[
S^e(z^{-1/2},z^{1/2})=S^o(z^{-1/2},z^{1/2})=\frac{z^{-1}}{1-2z^2/(1-2z)}=:S(z),
\]
which is valid (as a converging sum) for $z\in(0,\frac1{2}(\sqrt{3}-1))$. Then $S(z)^n$ is the series of $S$-admissible tandem walks with $n$ SE steps, starting from $(0,2)$ with a SE step, where every walk ending at $(i,j+2)$ is further weighted by $z^{(j-i)/2}$. Every walk ending at $(1,1)$ has contribution $z^{-1}$ to this series,   
  hence $s_n'\leq zS(z)^{n+1}$. We find $S'(z)=0$ for $z_0=1/4$, with $S(z_0)=16/3$. Hence $s_n'\leq 1/4\cdot (16/3)^{n+1}$, and $s_n=s'_n+2s_{n-1}'+s_{n-2}'\leq 2\cdot (16/3)^{n}$.
\end{proof}

\begin{thm}[growth rates]\label{thm:growth_rates}
The number $p_n$ of polyhedral orientations with $n$ inner vertices satisfies $\lim_{n\to\infty} p_n^{1/n}=9/2$, 
and the number $s_n$ of Schnyder labelings with $n$ inner faces satisfies $\lim_{n\to\infty} s_n^{1/n}=16/3$.
\end{thm}

\begin{proof}[Proof for $p_n$.] 
We use  the notation introduced in the proof of Proposition~\ref{prop:bound} for $p_n$ 
(in particular, $z_0=1/3$ and $S(z_0)=9/2$).   
Given Proposition~\ref{prop:bound}, it is enough to prove that $\lim \mathrm{inf}\ p_n^{1/n}\geq 9/2$. We let $\mue$ be the probability distribution on $\cSe$ where $\mue(i,j)=z_0^{(j-i)/2}/S(z_0)$ for $(i,j)\in\cSe$, and $\muo$ be the probability distribution on $\cSo$ where $\muo(i,j)=z_0^{(j-i)/2}/S(z_0)$ for $(i,j)\in\cSo$, and we define the \emph{$P$-admissible random walk} as the random walk (on $\mathbb{Z}^2$) starting from the origin where each step is drawn respectively under $\mue$ or $\muo$ depending on the current point having even or odd $y$. Let $P_n$ be the $P$-admissible random walk   formed by the first $n$ steps.  
For $\pi$ a $P$-admissible tandem walk of length $n$ starting at the origin and ending at $(i,j)$, we thus have $\mathbb{P}(P_n=\pi)=z_0^{j-i}/S(z_0)^n$ . 
A crucial point here is the following property: 

\medskip

\noindent {\bf (Box)}: ``The random walk $P_n$ stays asymptotically almost surely in the box $[- n^{2/3}, n^{2/3}]^2$." 

\medskip

In view of showing this property, let $\lambda(n,i,j)$ be the probability that $P_n$ ends at $(i,j)$, and let $\Lambda_n(x,y)=\sum_{i,j\in\mathbb{Z}^2}\lambda(n,i,j)x^iy^j$.  Let $F^e(t;x,y)=\sum_{n,i,j}a_{n,i,j}t^nx^iy^j$,   where $a_{n,i,j}$ (for $n\geq 0$ and $i,j\in\mathbb{Z}^2$) is the number of $P$-admissible tandem walks of length $n$ (on $\mathbb{Z}^2$)  starting at any fixed even point $(i_0,j_0)$ (i.e., with both $i_0,j_0$ even) and ending at $(i_0+i,j_0+j)$ (note that $F^e(t;x,y)$ does not depend on $(i_0,j_0)$ since we have no domain restriction). Similarly, let  $F^o(t;x,y)=\sum_{n,i,j}b_{n,i,j}t^nx^iy^j$, where $b_{n,i,j}$ (for $n\geq 0$ and $i,j\in\mathbb{Z}^2$) is the number of $P$-admissible tandem walks of length $n$ starting at any fixed odd point $(i_0,j_0)$ (i.e., with both $i_0,j_0$ odd) 
and ending at $(i_0+i,j_0+j)$. Thus, $\hat{F}^e(t;x,y):=F^e(t;x/z_0^{1/2},yz_0^{1/2})$ is the series 
of $P$-admissible walks starting at the origin, where each walk ending at $(i,j)$ is weighted by $z_0^{(j-i)/2}$. Consequently, 
\[
\Lambda_n(x,y) = \frac{[t^n]\hat{F}^e(t;x,y)}{[t^n]\hat{F}^e(t;1,1)}.
\]

 Let $\See\equiv\See(x,y)$ (resp. $\Seo\equiv\Seo(x,y)$) be the step-series for steps starting at even $y$ and ending at even (resp. odd) $y$, and let $\Soe\equiv\Soe(x,y)$ (resp. $\Soo\equiv\Soo(x,y)$) be the step-series for steps starting at odd $y$ and ending at even (resp. odd) $y$.
 The respective step-sets are $\cSee=\{(-2i,2j),\ i\geq 0,\ j\geq 1\}$, $\cSeo=\cSoe=\{(1,-1)\}\cup \{(-2i-1,2j+1),\ i\geq 0,\ j\geq 0\}$, and $\cSoo=\{(-2i,2j),\ i\geq 1,\ j\geq 0\}$, giving
 \[
 \See=\frac{1}{1-\bx^2}\frac{y^2}{1-y^2},\ \ \Soe=\Seo=x\by +\frac{\bx}{1-\bx^2}\frac{y}{1-y^2},\ \ \Soo=\frac{\bx^2}{1-\bx^2}\frac{1}{1-y^2}.
 \] 
Then, a first-step decomposition ensures that 
$F^e\equiv F^e(t;x,y)$ and $F^o \equiv F^o(t;x,y)$ satisfy the system
\begin{equation}
\left\{
\begin{array}{ll}
F^e &=\  1+t\, \See F^e + t\,\Seo F^o,\\
F^o &=\ 1+t\, \Soe F^e+ t\, \Soo F^o.
\end{array}
\right.
\end{equation} 

Solving this system, we find an explicit expression for 
$\hat{F}^e$ of the form 
\[
\hat{F}^e(t;x,y)=\frac{a_1(x,y)t+a_0(x,y)}{b_2(x,y)t^2+b_1(x,y)t+b_0(x,y)},
\] 
where $a_0(x,y),a_1(x,y),b_0(x,y),b_1(x,y),b_2(x,y)$ are 
polynomials in $x,y$, explicitly,

{\footnotesize
\begin{eqnarray*}
&a_0(x,y)=-3x^2y^4 + 9x^2y^2 + y^4 - 3y^2,\ \ a_1(x,y)=-9x^3y^3 + 27x^3y + 6xy^3 - 9xy - 3y^2,\\
&b_0(x,y)=-3x^2y^4 + 9x^2y^2 + y^4 - 3y^2,\ \ b_1(x,y)=-3x^2y^4 - 3y^2,\ \ b_2(x,y)=27x^4y^2 - 81x^4 - 27x^2y^2 + 27x^2.
\end{eqnarray*}
}

 The two roots (in $t$) of the denominator $b_2(x,y)t^2+b_1(x,y)t+b_0(x,y)$ are of the form 
\[1/\gamma^{\pm}(x,y)=\frac{x^2y^4+y^2\mp \sqrt{\Delta(x,y)}}{18x^2(1-3x^2+x^2y^2-y^2)},\] with $\Delta(x,y)$ a polynomial in $x,y$  such that $\Delta(1,1)=100>0$, explicitly,
{\footnotesize \[ 
\Delta(x,y)=36x^6y^6 + x^4y^8 - 216x^6y^4 - 48x^4y^6 + 324x^6y^2 + 216x^4y^4 + 14x^2y^6 - 216x^4y^2 - 48x^2y^4 + 36x^2y^2 + y^4.
\] }

Then, a partial 
fraction decomposition ensures that $\hat{F}^e(t;x,y)$ is of the form
\[
\hat{F}^e(t;x,y)=\frac{a_+(x,y)}{1-t\gamma_+(x,y)}+\frac{a_-(x,y)}{1-t\gamma_-(x,y)},
\] 
with $a_{\pm}(x,y)=-\frac{a_1(x,y)\gamma^{\pm}(x,y)+a_0(x,y)\gamma^{\pm}(x,y)^2}{2b_2(x,y)+b_1(x,y)\gamma^{\pm}(x,y)}$. Hence, 
\[
[t^n]\hat{F}(t;x,y)=a_+(x,y)\gamma_+(x,y)^n + a_-(x,y)\gamma_-(x,y)^n.
\]
 Since $a_+(1,1)=1>0$, and $\gamma_-(1,1)=-3$ while $\gamma_+(1,1)=9/2$ (so $|\gamma_-(1,1)|<\gamma_+(1,1)$), we have $[t^n]\hat{F}(t,x,y)=\Theta(\gamma_+(x,y)^n)$ valid for all $n$ and in a real neighbourhood $U$ of $(1,1)$, i.e., there are positive constants $C,C'$ such that, for all $(x,y)\in U$ and $n\geq 0$,  
 $C\,\gamma_+(x,y)^n \leq [t^n]\hat{F}(t;x,y) \leq C'\gamma_+(x,y)^n$.  
 Hence, in a real neighbourhoor of $(1,1)$, 
 \[
 \Lambda_n(x,y)  =\Theta\Big(\big(\gamma_+(x,y)/\gamma_+(1,1)\big)^n\Big). 
\]

We now look at the expansion of $g(x,y):=\gamma_+(x,y)/\gamma_+(1,1)$ at $(1,1)$. 
We find $\partial g/\partial x(1,1)=\partial g/\partial y(1,1)=0$. Since $g(x,y)$ is smooth at $(1,1)$, this ensures that $g(e^r,e^s)=1+O(r^2+s^2)$  in a real neighbourhood of $(0,0)$.  Now if $(X_n,Y_n)$ denotes the endpoint of $P_n$, then we have for every fixed $r\in\mathbb{R}$, 
\[
\mathbb{E}(e^{rX_n/\sqrt{n}})=\Lambda_n(e^{r/\sqrt{n}},1)=\Theta((1+O(1/n))^n)=O(1).
\]
In particular $\mathbb{E}(e^{X_n/\sqrt{n}})=O(1)$, and thus  by Markov's inequality there exists a constant $C>0$ such that 
$\mathbb{P}(X_n\geq x\sqrt{n})\leq Ce^{-x}$ for all $n\geq 1$ and $x>0$. By the union bound, the probability that $P_n$ visits a point of abscissa larger than $n^{2/3}$ is 
at most $\sum_{k=1}^n \mathbb{P}(X_k\geq n^{2/3})$. For $k\in[1..n]$, we have  $\mathbb{P}(X_k\geq n^{2/3})=\mathbb{P}(X_k\geq x\sqrt{k})$, where $x=\sqrt{n/k}n^{1/6}\geq n^{1/6}$. Hence,  
$\mathbb{P}(X_n\geq x\sqrt{n})=O(n\exp(-n^{1/6}))=o(1)$.  Similarly, using $r=-1$, one finds that the probability that $P_n$ visits a point of abscissa smaller than $-n^{2/3}$ is $o(1)$, and the same goes (considering $\Lambda_n(1,e^{\pm 1/\sqrt{n}})$) for the probability that $P_n$ visits a point of ordinate outside $[-n^{2/3},n^{2/3}]$. This concludes the proof of the property {\bf (Box)}. 

We let $\cB_n$ be the set of $P$-admissible tandem walks of length $n$ that start at the origin and stay in the box $[-n^{2/3},n^{2/3}]^2$. 
Recall that for $\pi$ a $P$-admissible walk of length $n$ ending at $(i,j)$, the probability that $P_n$ equals $\pi$ is $z_0^{(j-i)/2}/S(z_0)^n$. Hence every walk in $\cB_n$ has probability at most $(2/9)^n z_0^{-n^{2/3}}=(2/9)^{n+o(n)}$. Since the Box property ensures that the sum of probabilities of walks in $\cB_n$ is $1-o(1)$ we conclude that $|\cB_n|\geq (9/2)^{n+o(n)}$. It is now easy to conclude.  Letting $\cP_n$ be the set of $P$-admissible walks of length $n-1$ staying in the quadrant, starting at the origin and ending on the $x$-axis, and letting $n'$ be the smallest even integer that is at least $n^{2/3}$, and $k_n=3n'+8$, we have an injection $\iota$ from $\cB_{n-k_n}$ to $\cP_n$ that works as follows. 
Given $\pi\in\cB_{n-k_n}$ we construct $\iota(\pi)$ as the walk made of the step $(0,2n'+4)$, followed by $(n'+2)$ SE steps, at which point we attach the walk $\pi$ (since $\pi\in\cB_{n-k_n}$ and we start it at $(n'+2,n'+2)$ we strictly stay in the quadrant) to reach a point $(i,j)\in[2,2n'+2]^2$, from which we can canonically end the walk (we have to use $k_n-4-n'$ steps and end on the $x$-axis): if $j$ has same parity at $k_n-5-n'$ we append the step $(-2,k_n-5-n'-j)$ (by definition of $k_n$ we have $k_n-5-n'-j\geq 1$)  followed by $k_n-5-n'$ SE steps, otherwise we append the step $(-1,k_n-5-n'-j)$ followed by $k_n-5-n'$ SE steps. Thus, we have $p_n\geq  |\cB_{n-k_n}|\geq (9/2)^{n-k_n+o(n-k_n)}=(9/2)^{n+o(n)}$, so that $\lim \mathrm{inf}\ p_n^{1/n}\geq 9/2$.
\end{proof}

\begin{proof}[Proof for $s_n$.]
It follows the same lines, with the additional ingredient that face-steps have to be aggregated (since $S$-admissible tandem walks are counted with respect to the number of SE steps). The $S$-admissible tandem walks considered
here start with a SE step and are assumed to have their steps aggregated into groups formed by a SE step followed by a (possibly empty) sequence of face-steps (the original $S$-admissible walk stays in the quadrant iff the walk with aggregated steps stays in the shifted quadrant $\{x\geq 0, y\geq 1\}$).  
We use the notation of the proof of Proposition~\ref{prop:bound} for $s_n$ (in particular, $z_0=1/4$ and $S(z_0)=16/3$, and the notation $S^e(x,y)$ and $S^o(x,y)$).   
 We define $\mue$ to be the distribution on $\mathbb{Z}^2$ such that   $\mue(i,j)=a_{i,j}z_0^{(j-i)/2}/S(z_0)$, where $a_{i,j}=[x^iy^j]S^e(x,y)$,  and $\muo$ to be the distribution on $\mathbb{Z}^2$ such that  $\muo(i,j)=b_{i,j}z_0^{(j-i)/2}/S(z_0)$, where $b_{i,j}=[x^iy^j]S^o(x,y)$.  We define the \emph{random $S$-admissible (aggregated) tandem walk} 
 as the random walk (on $\mathbb{Z}^2$) starting  at the origin, where
 the next step is drawn under $\mue$ or $\muo$ depending on whether the current point has even or odd $y$. Letting 
   $\lambda(n,i,j)$ be the probability that the walk is at $(i,j)$ after $n$ steps, we define
  \[
  \Lambda_n(x,y):=\sum_{i,j\in\mathbb{Z}^2}\lambda(n,i,j)x^iy^j,
  \]
  Then, defining $a_{n,i,j}$ (resp. $b_{n,i,j}$) to be the total weight of  $S$-admissible tandem walks with $n$ aggregated steps,  starting at any even (resp. odd) point $(i_0,j_0)$  and ending at $(i_0+i,j_0+j)$, and defining $F^e\equiv F^e(t;x,y):=\sum_{n,i,j}a_{n,i,j}t^nx^iy^j$ and $F^o\equiv F^o(t;x,y):=\sum_{n,i,j}b_{n,i,j}t^nx^iy^j$, we have
  \[
  \Lambda_n(x,y)=\frac{[t^n]\hat{F}^e(t;x,y)}{[t^n]\hat{F}^e(t;1,1)},
  \]
  where $\hat{F}^e(t;x,y):=F^e(t,x/z_0^{1/2},yz_0^{1/2})$. 
  Moreover, the functions $F^e,F^o$ are specified by the system
  \[
  \left\{
  \begin{array}{ll}
  F^e&=1+tx\by \big( A^{oe}F^e + A^{oo}F^o \big),\\
    F^o&=1+tx\by \big( A^{ee}F^e + A^{eo}F^o \big),
  \end{array}
  \right.
  \]
  where $A^{ee}\equiv A^{ee}(x,y)$ (resp. $A^{eo}\equiv A^{eo}(x,y)$) is the series of aggregated face-steps (not including an initial SE step) starting at even $y$ and ending at even (resp. odd) $y$, and $A^{oe}\equiv A^{oe}(x,y)$ (resp. $A^{oo}\equiv A^{oo}(x,y)$) is the series of aggregated face-steps (not including an initial SE step) starting at odd $y$ and ending at even (resp. odd) $y$.   
  Using a first-step decomposition, these functions are themselves specified by systems in terms of the rational series $\tilde{S}^{ee},\tilde{S}^{eo},\tilde{S}^{oe},\tilde{S}^{oo}$ considered in the proof of Proposition~\ref{prop:bound}:
  \[
  \left\{
  \begin{array}{ll}
  A^{ee}&=1+\tilde{S}^{ee}A^{ee} + \tilde{S}^{eo}A^{oe}\\
  A^{oe}&=\tilde{S}^{oe}A^{ee} + \tilde{S}^{oo}A^{oe}\\
  \end{array}
  \right.,\ \ \   
  \left\{
  \begin{array}{ll}
  A^{eo}&=\tilde{S}^{ee}A^{eo} + \tilde{S}^{eo}A^{oo}\\
  A^{oo}&=1+\tilde{S}^{oe}A^{eo} + \tilde{S}^{oo}A^{oo}\\
  \end{array}
  \right. ,
  \]
  which yields
  \[
  A^{ee}=A^{oo}=\frac{x^2y^2 - x^2 + y^2 + 1}{x^2y^2 - x^2 + 2y^2 + 1},\ \ A^{eo}=\frac{-xy^3}{x^2y^2 - x^2 + 2y^2 + 1},\ \ A^{oe}=\frac{-y}{x(x^2y^2 - x^2 + 2y^2 + 1)}.
  \]
  The system for $F^e,F^o$ can then be solved, yielding an expression for $\hat{F}^e$ of the form $\hat{F}^e(t;x,y)=\frac{a_1(x,y)t+a_0(x,y)}{b_2(x,y)t^2+b_1(x,y)t+b_0(x,y)}$ where $a_0(x,y),a_1(x,y),b_0(x,y),b_1(x,y),b_2(x,y)$ are 
polynomials in $x,y$, explicitly,
{\small
\begin{eqnarray*}
&a_0(x,y)= 8x^2y^2-2x^2y^4 - y^4 - 2y^2,\ \ a_1(x,y)=32x^3y-8x^3y^3 - 2x^2y^4  - 2xy^3 - 8xy,\\
&b_0(x,y)= 8x^2y^2-2x^2y^4  - y^4 - 2y^2 ,\ \ b_1(x,y)=-2x^2y^4 - 2y^2,\ \ b_2(x,y)=32x^4y^2 - 128x^4 + 32x^2.
\end{eqnarray*}
}

The two roots (in $t$) of the denominator $b_2(x,y)t^2+b_1(x,y)t+b_0(x,y)$ are of the form 
\[1/\gamma^{\pm}(x,y)=\frac{x^2y^4+y^2\mp \sqrt{\Delta(x,y)}}{32x^2(x^2y^2 - 4x^2 + 1)},\] with $\Delta(x,y)$ a polynomial in $x,y$ such that $\Delta(1,1)=196>0$, explicitly,
 {\small
 \[
\Delta(x,y)= 64x^6y^6 + x^4y^8 - 512x^6y^4 + 32x^4y^6 + 1024x^6y^2 + 2x^2y^6 - 512x^4y^2 + 32x^2y^4 + 64x^2y^2 + y^4.
 \] 
 }
 
 Then, a partial 
fraction decomposition ensures that $\hat{F}^e(t;x,y)$ is of the form
\[
\hat{F}^e(t;x,y)=\frac{a_+(x,y)}{1-t\gamma_+(x,y)}+\frac{a_-(x,y)}{1-t\gamma_-(x,y)},
\] 
with $a_{\pm}(x,y)=-\frac{a_1(x,y)\gamma^{\pm}(x,y)+a_0(x,y)\gamma^{\pm}(x,y)^2}{2b_2(x,y)+b_1(x,y)\gamma^{\pm}(x,y)}$. Hence, 
\[
[t^n]\hat{F}(t;x,y)=a_+(x,y)\gamma_+(x,y)^n + a_-(x,y)\gamma_-(x,y)^n.
\]
 Since $a_+(1,1)=1>0$, and $\gamma_-(1,1)=-4$ while $\gamma_+(1,1)=16/3$ (so $|\gamma_-(1,1)|<\gamma_+(1,1)$), we have $[t^n]\hat{F}(t,x,y)=\Theta(\gamma_+(x,y)^n)$ valid for all $n$ and in a real neighbourhood of $(1,1)$. Hence, in a real neighbourhoor of $(1,1)$,
 \[
 \Lambda_n(x,y)  =\Theta\Big(\big(\gamma_+(x,y)/\gamma_+(1,1)\big)^n\Big). 
\]
The rest of the argument is again very similar to the proof for $p_n$. We first check that the function $g(x,y):=\gamma_+(x,y)/\gamma_+(1,1)$ is such that $\partial g/\partial x(1,1)=\partial g/\partial y(1,1)=0$, which ensures that the $S$-admissible tandem walk restricted to the first $n$ aggregated steps stays a.a.s. in the box $[- n^{2/3},n^{2/3}]^2$.  Then an injective construction similar to the one used for $\cP_n$ ensures that $s_n'\geq (16/3)^{n+o(n)}$. 
\end{proof}

We recall from the proof of Theorem~\ref{thm:growth_rates} that $p_{n-1}$ is closely related to the number $\tilde{p}_n$
 of $P$-admissible quadrant tandem walks of length $n$ from the origin to $(1,1)$.  
With the notation in the proof of Theorem~\ref{thm:growth_rates}, recall that $(X_n,Y_n)$ is the endpoint 
of the $P$-admissible random walk $P_n$, and that the probability that $P_n$ stays in the quadrant and ends at $(1,1)$ is $(2/9)^n\tilde{p}_n$. 
We have 
\[
\mathbb{E}(e^{rX_n+sY_n/\sqrt{n}})\sim g(e^{r/\sqrt{n}},e^{s/\sqrt{n}})^n,
\]
for some explicit function $g(e^r,e^s)=1+\frac1{2}a(r^2+s^2)+brs+o(r^2+s^2)$, where we find $a=72/5,b=-81/10$. A central limit theorem follows for $(X_n,Y_n)$, with zero drift, and with covariance matrix $\Sigma=\left(\begin{array}{ll}a&b\\ b&a\end{array}\right)$, that is, $(X_n/\sqrt{n},Y_n/\sqrt{n})$ converges in law to the normal 2d distribution that has zero mean and covariance matrix $\Sigma$. It is also expected that $P_n$ (renormalized by a factor $\sqrt{n}$) converges to a Brownian motion with the same covariance matrix.   
Results by Denisov and Wachtel~\cite[Theo.6]{denisov2015random} 
(to be extended to a bimodal setting in order to be applicable here) then indicate that the probability that $P_n$ stays in the 
quadrant and ends at $(1,1)$ should behave as $\kappa\, n^{-1-\pi/\arccos(\xi)}$, with $\kappa$ a positive constant, and $\xi=-b/a=9/16$. 

Similarly, for Schnyder labelings, there is a random walk $S_n$ of length $n$ starting at $(0,2)$ 
such that $4\cdot(3/16)^n\cdot s_{n-1}'$ is the probability that $S_n$ stays in the quadrant (actually, in the shifted quadrant $\{x\geq 0,y\geq 1\}$) and ends at $(1,1)$. As before, we find that a central limit theorem applies for the endpoint of $S_n$, again with zero drift and with a covariance matrix of the form $\left(\begin{array}{ll}a&b\\ b&a\end{array}\right)$, where this time $a=192/7,b=-1408/63$. Again it is expected that the probability that $S_n$ stays in the quadrant and ends at $(1,1)$ should behave as $\kappa'\, n^{-1-\pi/\arccos(\xi')}$, with $\kappa'$ a positive constant, and $\xi'=-b/a=22/27$. 

Based on this, we can conjecture the following: 

\begin{cjt}\label{cj:asympt}
We have $p_n\sim \kappa\cdot(9/2)^n\cdot n^{-\alpha}$, where $\alpha=1+\pi/\mathrm{arccos}(9/16)\approx 4.23$ and $\kappa$ is a positive constant; and  $s_n\sim \kappa'\cdot (16/3)^n\cdot n^{-\alpha'}$, where $\alpha'=1+\pi/\mathrm{arccos}(22/27)\approx 6.08$ and $\kappa'$ is a positive constant. 
\end{cjt}

\begin{rk}
The two above-mentioned central limit results could also be established using the theory of Markov-modulated random walks~\cite{tang2008markov}, letting the two states of the underlying Markov chain correspond to the two types of lattice points (even-even coordinates or odd-odd coordinates). 
\end{rk}

\begin{rk}
Using the method in~\cite{bostan2014non} one can easily verify that the constants $\alpha$ and $\alpha'$ in Conjecture~\ref{cj:asympt} are irrational (e.g. for $\alpha$, we have to check that $X(z):=16z-9$ is such that 
$X(\frac1{2}(z+z^{-1}))=(8z^2 - 9z + 8)/z$ has no cyclotomic factor in its numerator, which is indeed the case since the numerator is irreducible of degree $2$ and  has coefficients of absolute value larger than~$2$).  
Hence, a corollary to  Conjecture~\ref{cj:asympt} would be that the generating functions of the sequences $p_n$ and $s_n$ are not D-finite. \dotfill
\end{rk}

\section{Relation to contact-systems of curves}\label{sec:contacts}

We briefly discuss the link between the structures considered here and contact-systems of curves. 
Let us first recall how plane bipolar orientations and transversal structures can be formulated as  bicolored contact-systems of curves. Using the operations shown in Figure~\ref{fig:contacts_local}(a), a plane bipolar orientation $X$ becomes a bicolored contact-system of  curves, i.e., a set of smooth curves that are either red (``horizontal" curves, arising from the vertices of $X$) or blue (``vertical" curves, arising from the faces of $X$, considering that there is a left outer face and a right outer face) and such that each contact is from the tip of a curve to a side of a curve of the opposite color (with the exception of the four ``outer" contacts where the four outer curves meet to form an outer ``rectangular" shape), see Figure~\ref{fig:contacts}(a) for an example. Such systems are to be considered up to continuous deformation, and the contacts along a side of curve can freely slide (the order of the contacts along a side of curve matters, but contacts that are on different sides of a curve are not comparable). In contrast, via duality, a transversal structure (or transversal edge-partition~\cite{fusy2009transversal}) is a bicolored contact-system with the same properties but with an additional rigidity: there is a total order of all the contacts along a given curve (not just an order on the contacts on each side of the curve); this is due to the fact that the contacts are really considered as trivalent vertices of a map, see Figure~\ref{fig:contacts}(b) for an example\footnote{Algorithms for finding a geometric representation of a bicolored contact-system where the red/blue curves are horizontal/vertical segments are developed in~\cite{Tamassia} for the ``free" case (based on plane bipolar orientations), and in~\cite{kant1997regular} for the ``rigid" case (based on transversal structures).}. 

\begin{figure}
\begin{center}
\includegraphics[width=14.6cm]{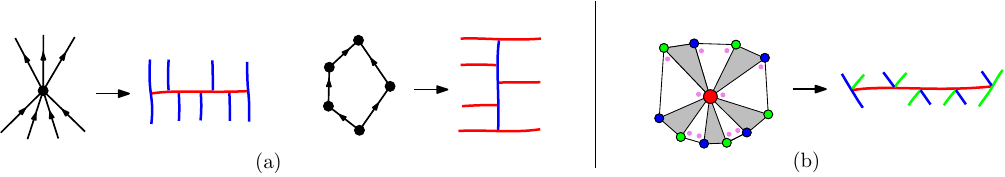}
\end{center}
\caption{(a) Local operations performed at vertices and faces of a plane bipolar orientation to yield a bicolored contact-system of curves. (b) Local operation performed at vertices of an Eulerian triangulation to yield a tricolored contact-system of curves.}
\label{fig:contacts_local}
\end{figure}

\begin{figure}
\begin{center}
\includegraphics[width=14.6cm]{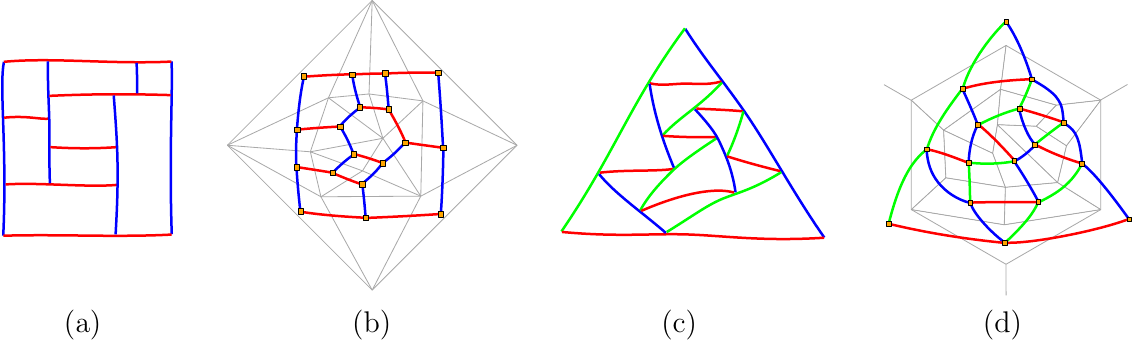}
\end{center}
\caption{(a) a free bicolored contact-systems (corresponding to a plane bipolar orientation), (b) a rigid bicolored contact-systems (dual to a transversal structure), (c) a free tricolored contact-systems (corresponding to a polyhedral orientation), (d) and a rigid tricolored contact-system (dual to a Schnyder labeling).}
\label{fig:contacts}
\end{figure}

\begin{figure}
\begin{center}
\includegraphics[width=14.6cm]{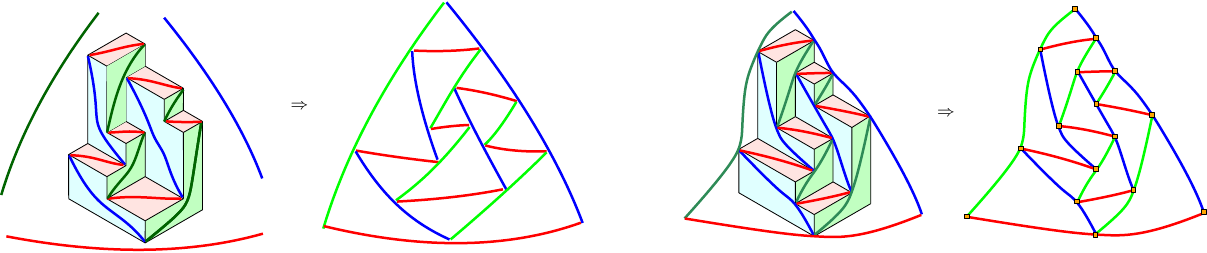}
\end{center}
\caption{Left: a corner polyhedron superimposed with the associated free tricolored contact-system. Right: a rigid orthogonal surface superimposed with the associated rigid tricolored contact-system. }
\label{fig:tricolored_superimposed}
\end{figure}

\begin{figure}
\begin{center}
\includegraphics[width=14.6cm]{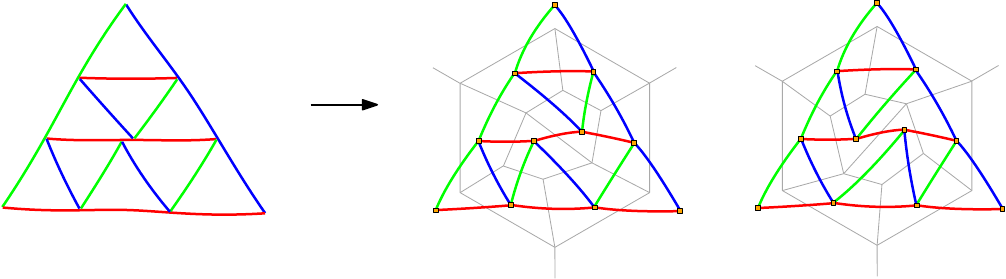}
\end{center}
\caption{On the left, the free tricolored contact-system from $\cP_{2,3,3}$ that yields two different rigid contact-systems (shown on the right).}
\label{fig:contacts_not_unique}
\end{figure}

On the other hand, the structures considered in this article are naturally associated to tricolored contact-systems, with  polyhedral orientations (resp. Schnyder labelings) playing the role of plane bipolar orientations (resp. of transversal structures)\footnote{As in the bicolored case, curves are considered up to deformation, the existence of a representation with curves as segments is studied in~\cite{felsner2020plattenbauten,gonccalves20193,gonccalves2012triangle}}.  Using the operation  shown in Figure~\ref{fig:contacts_local}(c), a polyhedral orientation becomes a system of curves that are red, blue or green, with contact points where the tips of two curves meet a side of another curve, such that the colors red, green, blue occur in ccw order around the contact (as for bicolored systems, there are three special contacts in the outer face, where the three outer curves meet to form a ``triangular" outer shape), see Figure~\ref{fig:contacts}(c) which shows the system associated with the polyhedral  orientation of Figure~\ref{fig:corner_poly} (the left part of Figure~\ref{fig:tricolored_superimposed} shows the correspondence starting from the corner polyhedron instead of the polyhedral orientation). As in the bicolored case, the contacts on each side of a given curve are free to slide, independently of the position of the contacts on the opposite side. On the other hand, via duality (and using Remark~\ref{rk:colors_sch}), 
a Schnyder labeling in $\tilde{\cS}$ is a similar contact-system, with the further rigidity that, for a given curve, there is a total order (along the curve) on the contacts-points (from both sides), see Figure~\ref{fig:contacts}(d) which gives the dual of the Schnyder labelling of Figure~\ref{fig:schnyder_lab} (the right part of Figure~\ref{fig:tricolored_superimposed} shows the correspondence starting from a rigid orthogonal surface instead of the Schnyder labelling). 
The condition that the three outer white vertices are non-isolated is necessary to avoid interior points of two outer curves to meet and form a pinch point. 

\begin{rk}
Regarding the sets and their relations, $\cP_n$ corresponds to the set of free tricolored contact-systems with $n+3$ curves, and (for $a+b+c=n$) $\cP_{a,b,c}$ corresponds to the subset of those with $a+1$ red curves, $b+1$ blue curves, and $c+1$ green curves. On the other hand,  $\tilde{\cS}_n$ corresponds to the set of rigid tricolored contact-systems with $n$ four-valent contact-points, equivalently, with $n+3$ curves (indeed, each four-valent contact-point is the tip of two 
inner curves, and each inner cuve has two tips that are four-valent); and (for $a+b=n+2$) $\tilde{S}_{a,b}$ corresponds to the subset of those having $a+1$ dark faces and $b-1$ light faces (including the outer one) for the canonical face-bicoloration of the underlying Eulerian map. Under the formulation as contact-systems, we clearly have a surjective  mapping from $\tilde{\cS}_n$ to $\cP_n$. Since $\tilde{s}_n\geq p_n$ starts to be strictly larger from $n=8$, the smallest free contact-system that lifts to more than $1$ rigid contact-systems has size $n=8$. The one in $\cP_{2,3,3}$ is shown in Figure~\ref{fig:contacts_not_unique} (by rotation, there is also one in $\cP_{3,2,3}$ and one in $\cP_{3,3,2}$; there is no other one at size $8$, since $\tilde{s}_8-p_8=3$).   
\dotfill
\end{rk}

\medskip

\noindent{\emph{Acknowledgements.} The authors are grateful to the anonymous referee for very helpful suggestions. They also thank Olivier Bernardi, Wenjie Fang, Shizhe Liang, and Kilian Raschel for interesting discussions and bibliographic pointers. 
The authors are partially supported by the project ANR-16-CE40-0009-01 (GATO) and the project ANR-20-CE48-0018 (3DMaps). \'EF is also partially supported by the project ANR19-CE48-011-01 (COMBIN\'E).




\bibliographystyle{99}
\bibliography{biblio_fpsac}

\end{document}